\documentclass[preprint,10.5pt]{elsarticle}
\usepackage{lipsum}

\usepackage[utf8]{inputenc}
\usepackage{graphicx}
\usepackage[usenames,dvipsnames]{color}
\usepackage[titletoc]{appendix}
\usepackage{natbib}
\usepackage[nottoc]{tocbibind}
\usepackage[protrusion=true,expansion=true]{microtype}
\usepackage{mathtools}
\usepackage{amsmath}

\usepackage{amsfonts}
\usepackage{amssymb}
\usepackage{graphicx}
\usepackage{mathrsfs}
\usepackage{xfrac}
\usepackage{dsfont}
\usepackage{epsfig}
\usepackage{amsthm}
\usepackage{url}
\usepackage{float}
\usepackage{color}
\usepackage{longtable}
\usepackage{parskip}
\usepackage{caption}
\usepackage{subfig}
\usepackage{sidecap}
\usepackage{comment}
\usepackage{setspace}
\usepackage[margin=0.85in]{geometry}
\usepackage[font=onehalfspacing]{caption}
\usepackage[usenames,dvipsnames]{xcolor}
\usepackage{hyperref}
\usepackage[labelfont=bf]{caption}
\usepackage{tikz,lipsum}
\usepackage{pgfplots}
\usepackage{pdfpages}
\usepackage{subcaption}
\usepackage{todonotes}
\usepackage{cleveref}
\usepackage{theoremref}
\usepackage{bbold}

\pgfplotsset{width=10cm,compat=1.9}

\hypersetup{colorlinks=true,citecolor=blue,linkcolor=blue,urlcolor=blue}

  \newtheoremstyle{problem}
  {\topsep}
  {\topsep}
  {\itshape}
  {}
  {\bfseries}
  {.}
  { }
  {Problem MFT}

  \newtheoremstyle{problem1}
  {\topsep}
  {\topsep}
  {\itshape}
  {}
  {\bfseries}
  {.}
  { }
  {Problem 1}

  \newtheoremstyle{problem2}
  {\topsep}
  {\topsep}
  {\itshape}
  {}
  {\bfseries}
  {.}
  { }
  {Problem 2}

  \newtheoremstyle{condition1}
  {\topsep}
  {\topsep}
  {\itshape}
  {}
  {\bfseries}
  {.}
  { }
  {(H1)}

  \newtheoremstyle{condition2}
  {\topsep}
  {\topsep}
  {\itshape}
  {}
  {\bfseries}
  {.}
  { }
  {(H2)}
  
\theoremstyle{definition}
\newtheorem{definition}{Definition}

  \newtheorem{theorem}{Theorem}
  \newtheorem{corollary}{Corollary}
  \newtheorem{proposition}{Proposition}
  \newtheorem{lemma}{Lemma}
  \newtheorem{corollary
}{Corollary}
  \newtheorem{assumption}{Assumption}
  \newtheorem{remark}{Remark}

\theoremstyle{problem}
  \newtheorem{probMFT}{Problem MFT}

\theoremstyle{problem1}
  \newtheorem{prob1}{Problem 1}

\theoremstyle{problem2}
  \newtheorem{prob2}{Problem 2}

\theoremstyle{condition1}
  \newtheorem{cond1}{(H1)}

\theoremstyle{condition2}
  \newtheorem{cond2}{(H2)}

\usepgfplotslibrary{groupplots}

  \newcommand{\E}{\mathbb{E}}

\begin{document}
\makeatletter
\def\ps@pprintTitle{
  \let\@oddhead\@empty
  \let\@evenhead\@empty
  \let\@oddfoot\@empty
  \let\@evenfoot\@oddfoot
}
\makeatother

\begin{frontmatter}

\title{A Decomposition Method for LQ Conditional McKean-Vlasov Control Problems with Random Coefficients\footnote{Funding: Dena Firoozi would like to acknowledge the support of the Natural Sciences and Engineering Research Council of Canada (NSERC), grants RGPIN-2022-05337 and DGECR-2022-00468. Shuang Gao would like to acknowledge the support of the NSERC grant RGPIN-2024-06612 and   the Fonds de recherche du Québec – Nature et technologies (FRQNT) grant 359699 (https://doi.org/10.69777/359699).  Onésime Hounkpe would like to acknowledge the support of Optimizing Power Skills in Interdisciplinary, Diverse \& Innovative Academic Networks (OPSIDIAN) from Polytechnique Montréal.}} 

\author[HEC]{Onésime Hounkpe}  
\author[UT]{Dena Firoozi\footnote{Corresponding Author.}}       
\author[Poly]{Shuang Gao} 

\address[HEC]{Department of Decision Sciences, HEC Montréal, Montreal, QC, Canada. \\Email: onesime.hounkpe@hec.ca}    

\address[UT]{Department of Statistical Sciences, University of Toronto, Toronto, ON, Canada. \\Email: dena.firoozi@utoronto.ca}

\address[Poly]{Department of Electrical Engineering, Polytechnique Montréal, QC, Canada. \\Email: shuang.gao@polymtl.ca} 

\date{}

\begin{abstract}
We propose a decomposition method for solving a general class of linear-quadratic (LQ) McKean-Vlasov control problems involving conditional expectations and random coefficients, where the system dynamics are driven by two independent Wiener processes. Unlike existing approaches in the literature for these problems, such as the extended stochastic maximum principle and the extended dynamic programming methods, which often involve additional technical complexities and sometimes impose restrictive conditions on control inputs, our approach decomposes the original McKean-Vlasov control problem into two decoupled stochastic optimal control problems, one of which has a constrained admissible control set. These auxiliary problems can be solved using classical methods. We establish an equivalence between the well-posedness and solvability of the auxiliary problems and those of the original problem, and show that the sum of the optimal controls of the auxiliary problems yields the optimal control of the original problem. Moreover, by applying a variational method, we characterize the optimal solution to the McKean-Vlasov control problem via two decoupled sets of (non-McKean-Vlasov) linear forward-backward stochastic differential equations (FBSDEs), each corresponding to one of the auxiliary problems. Finally, we show that standard dynamic programming can also be applied to solve the resulting auxiliary problems.
\end{abstract}

\begin{keyword}                           
McKean-Vlasov control, Decomposition method, Stochastic coefficients.               
\end{keyword} 

\end{frontmatter}

\section{Introduction}

McKean-Vlasov control problems, also known as mean-field type or mean-field control problems, are characterized by state dynamics and, in some cases, cost functionals that depend on the distributions of either the state, the control, or both \citep{Andersson2011,Buckdahn2011,Yong2013,Carmona2013,Lauriere2014,Pham:DetCoef2017,Pham:RandCoef2016,Acciaio2019,Bensoussan2025,Carmona2025,Pham2026}. A key motivation for studying these problems arises from the analysis of infinite-population cooperative mean-field games \citep{Carmona2013,Pham:RandCoef2016}. McKean-Vlasov control problems have various applications, including systemic risk modeling \citep{Lauriere2014,Pham:RandCoef2016},  exhaustible resources production \cite{Graber2016}, mean-variance portfolio selection \citep{Andersson2011,Pham:RandCoef2016}, and environmental risk management \citep{Bensoussan2025}.

While the classical stochastic maximum principle (see, e.g.,  \citep{Peng1990,XYZ2012}) and stochastic dynamic programming (see, e.g.,  \cite{XYZ2012,Phame2009}) 
do not apply directly to McKean-Vlasov control problems \citep{Lauriere2014,Carmona2015,Tang2019}, both frameworks have been extended to address such problems. Specifically, the extended stochastic maximum principle approach \citep{Andersson2011,Buckdahn2011,Carmona2015,Acciaio2019} gives rise to a Hamiltonian system involving McKean-Vlasov forward-backward stochastic differential equations (FBSDEs). Moreover, the extended dynamic programming approach \citep{Lauriere2014,Bensoussan2015,Pham:NoCN2018,Pham:DetCoef2017} reformulates McKean–Vlasov control problems into ones where the density or distribution of the original state is treated as the new controlled state. When the density is treated as the new  state \cite{Lauriere2014,Bensoussan2015}, the Fokker-Planck equation is used to characterize the density evolution. In the case where the marginal distribution is treated as the new state \cite{Pham:NoCN2018,Pham:DetCoef2017,Pham:RandCoef2016}, 
the value function is defined on a Wasserstein space of probability measures. However, this method imposes certain restrictions, for instance in models with dynamics driven by two Wiener processes, where the diffusion and drift coefficients are adapted to one of these noise processes, and the admissible control inputs are required to be adapted only to that noise \cite{Pham:DetCoef2017}. 

Linear-quadratic (LQ) McKean-Vlasov control problems, where linear dynamics and quadratic cost functionals depend on (conditional) expectations of the state, the control, or both, have also been extensively studied in the literature. In this framework, the extended stochastic maximum principle leads to a system of linear McKean-Vlasov FBSDEs that involve the (conditional) expectations of the adjoint process, as well as those of other processes appearing in the model \cite{Yong2013,Graber2016,Sun2017,Tang2019,Yang2022,Xiong2025}. Moreover, the extended dynamic programming approach imposes similar restrictions, as noted above, on the admissible control inputs in cases where the dynamics are driven by two Wiener processes \citep{Pham:RandCoef2016}. Additionally, the global well-posedness of infinite-dimensional Hamilton-Jacobi equations for a class of LQ McKean-Vlasov control problems with nonlinear, and possibly nonconvex, dependence on the conditional expectations of the state and the control given the common noise is investigated in \cite{Li2025}.

In this paper, we propose an alternative approach to address a general class of LQ McKean-Vlasov control problems where dynamics are driven by two independent Wiener processes and involve the conditional expectation of the state with respect to one of them, with the coefficients also adapted to this same noise. The key idea is to decompose the original problem into two auxiliary stochastic optimal control problems: one that is independent of the conditional expectations of the relevant processes, and the other that involves a constrained admissible control set requiring the conditional expected value of the control process to be almost surely zero. These two problems may be solved using standard methods developed for non-McKean-Vlasov problems. We establish the equivalence of well-posedness, solvability and solutions between the original Mckean-Vlasov problem and the two auxiliary stochastic control problems. Moreover, we employ a variational method to solve these auxiliary problems and, as a result, express the solution to the original Mckean-Vlasov control problem as the sum of these solutions. This solution is fully characterized by two decoupled sets of (non-Mckean-Vlasov) linear FBSDEs, each of which characterizes the solution to one of the auxiliary problems. Finally, we show that standard dynamic programming can also be applied to solve the resulting auxiliary problems. The proposed decomposition method is inspired by \citep{Mahajan2015}, which studies discrete-time $N$-player cooperative games where agents are coupled through their average state, and it naturally applies to simpler LQ Mckean-Vlasov models involving a single Wiener process or deterministic coefficients. 

Compared to the extended stochastic maximum principle \cite{Yong2013,Graber2016,Sun2017,Tang2019,Yang2022,Xiong2025}, our proposed method yields simpler FBSDEs that do not explicitly depend on the conditional expectations of the relevant processes. Moreover, relative to the extended dynamic programming approach \cite{Pham:RandCoef2016}, our method relaxes the restrictions on the admissible control set by allowing controls to be adapted to both Wiener processes driving the dynamics, while avoiding the need to treat infinite-dimensional Hamilton-Jacobi equations as in \cite{Li2025}. Finally, it addresses the open problem indicated in  \citep{Yong2013},  motivated by observations on the potential use of LQ problems with constrained control sets in studying LQ Mckean-Vlasov optimal control problems. It is worth noting that, unlike \cite{Mei2024, Mei2025, Yong2013}, where the McKean-Vlasov control problem is treated as a single problem and an orthogonal decomposition of the state is used to facilitate the derivation or representation of the optimal control, we use the decomposition to split, from the outset, the original McKean-Vlasov problem into two separate problems, and then solve them independently. To the best of our knowledge, the systematic establishment of equivalence in well-posedness, solvability, and solutions between the original McKean-Vlasov system and the two auxiliary stochastic control problems does not appear in the prior literature.

The remainder of the paper is organized as follows. \Cref{Sec2} introduces the necessary notation and preliminaries, and presents the formulation of the McKean-Vlasov control problems under consideration. \Cref{Sec3} details the proposed decomposition approach and applies a variational method to derive the solution to the McKean-Vlasov control problems under study, while the use of dynamic programming is deferred to \ref{AppendixB}. Specifically, the main result of the paper is stated in \Cref{KeyTheo}, which establishes the equivalence of optimal pairs between the considered class of LQ conditional McKean-Vlasov control problems and two independent classes of standard (non-McKean-Vlasov) LQ control problems. This theorem relies on \Cref{StateSplit}, \Cref{lemma:cost-term-decomp} and \Cref{CostSplit}. Subsequently, \Cref{UnifConvChara} establishes uniform convexity for all three problems as a sufficient condition for unique solvability. Finally, \Cref{OptimPrincMFT}, \Cref{optimal-control-MFT} and \Cref{thm:feedback-rep} together characterize the optimal control of the McKean-Vlasov control problem as the sum of the optimal controls of the two auxiliary problems. \Cref{sec:conclusion} concludes the paper. 

\section{Problem Formulation}\label{Sec2}
We begin by introducing the necessary notation. Two independent one-dimensional Wiener processes $W^{0}:=\left(W^{0}_{t}\right)_{t\in[0,T]}$ and $W:=\left(W_{t}\right)_{t\in[0,T]}$ are defined on the probability space $\left(\Omega,\mathfrak{F},\mathbb{P}\right)$. The filtration generated by $W^{0}$ is denoted by $\mathcal{F}^{0}:=\left(\mathcal{F}^{0}_{t}\right)_{t\in[0,T]}$, and the one generated by both $W^{0}$ and $W$ is denoted by $\mathcal{F}:=\left(\mathcal{F}_{t}\right)_{t\in[0,T]}$. These filtrations satisfy the usual conditions. For an arbitrary Banach space $\mathbb{X}$, endowed with the norm $\Vert \cdot\Vert_{\mathbb{X}}$ and an arbitrary filtration $\mathcal{G}$, the following spaces of random vectors or stochastic processes are defined:
\begin{align*}
        &L^{2}_{\mathcal{G}_{0}}\left(\Omega,\mathbb{X}\right):= \left\{Z:\Omega \to \mathbb{X}  \, | \, Z \mbox{ is } \,  \mathcal{G}_{0}-\mbox{measurable and}\, \E\Vert Z\Vert^{2}_{\mathbb{X}}<\infty\right\},\allowdisplaybreaks\\        &\mathcal{C}_{\mathcal{G}}\left([0,T],\mathbb{X}\right):=\left\{h:[0,T]\times\Omega \to \mathbb{X} \,  |\,  h(\cdot)\mbox{ is } \,  \mathcal{G}-\mbox{progessively measurable, continuous and }\Vert h \Vert_{\mathcal{C}_{\mathcal{G}}}<\infty \right\},\allowdisplaybreaks\\
        &L_{\mathcal{G}}^{2}\left( [0,T],\mathbb{X}\right):=\left\{h:[0,T]\times\Omega \to \mathbb{X} \,  |\,h(\cdot)\mbox{ is } \, \mathcal{G}-\mbox{progessively measurable and }\Vert h \Vert_{L_{\mathcal{G}}^{2}}<\infty \right\},\allowdisplaybreaks\\
        &L_{\mathcal{G}}^{\infty}\left( [0,T],\mathbb{X}\right):=\left\{h:[0,T]\times\Omega \to \mathbb{X} \,  |\,  h(\cdot) \mbox{ is }\, \mathcal{G}-\mbox{progessively measurable and }\Vert h \Vert_{L_{\mathcal{G}}^{\infty}}<\infty \right\},\allowdisplaybreaks\\
        &L_{\mathcal{G}_{\tau}}^{\infty}\left(\mathbb{X}\right):= \left\{Z:\Omega \to \mathbb{X} \,  |\,  Z\mbox{ is }\,  \mathcal{G}_{\tau}-\mbox{measurable and }\Vert Z \Vert_{L_{\mathcal{G}_{\tau}}^{\infty}}<\infty \right\},\,\,\mbox{for } \tau\geq 0,\,
\end{align*}
where $\Vert h \Vert_{\mathcal{C}_{\mathcal{G}}}:= \left(\E\sup_{t\in[0,T]}\Vert h_{t}\Vert_{\mathbb{X}}^{2}\right)^{\frac{1}{2}}$, $\Vert h \Vert_{L_{\mathcal{G}}^{2}}:= \left(\E\int_{0}^{T}\Vert h_{t}\Vert_{\mathbb{X}}^{2}dt\right)^{\frac{1}{2}}$, $\Vert Z \Vert_{L_{\mathcal{G}_{\tau}}^{\infty}}:=\inf\left\{\alpha\geq0\, |\,  \Vert Z(\omega)\Vert_{\mathbb{X}}\leq \alpha, d\mathbb{P}- a.s.\right\}$ and $\Vert h \Vert_{L_{\mathcal{G}}^{\infty}}:=\inf\left\{\alpha\geq0\, |\,  \Vert h(t,\omega)\Vert_{\mathbb{X}}\leq \alpha,\,  dt\otimes d\mathbb{P}-a.s.\right\}$. Moreover, the set of all real-valued symmetric matrices of dimension $n\times n$ are denoted by $\mathbb{S}^{n}$.

We study a class of conditional McKean-Vlasov control problems with stochastic coefficients, defined as follows. 
\begin{probMFT}\thlabel{MTProb}
Define an admissible control set as $\mathcal{U}:  = L_{\mathcal{F}}^{2}\left( [0,T],\mathbb{R}^{d}\right)$. Find the optimal control process $u^{*} \in \mathcal{U}$ that minimizes the cost functional $J\left(u\right)$ given by      
            \begin{align}
                &J\left(u\right) = \frac{1}{2}\E\bigg\{\int_{0}^{T}\Big[\left(x_{t}-H\E[x_{t}|\mathcal{F}_{t}^0]\right)^{\intercal}Q_{t}\left(x_{t}-H\E[x_{t}|\mathcal{F}_{t}^0]\right) + 2\left(x_{t}-H\E[x_{t}|\mathcal{F}_{t}^0]\right)^{\intercal}S_{t}u_{t}+ u^{\intercal}_{t}R_{t}u_{t}  \notag\allowdisplaybreaks\\ 
                    &\qquad \qquad \qquad + 2\zeta^{\intercal}_{t}\left(x_{t}-H\E[x_{t}|\mathcal{F}_{t}^0]\right) + 2\varpi_{t}^{\intercal}u_{t} \Big]dt + \left(x_{T}-H\E[x_{T}|\mathcal{F}_{T}^0]\right)^{\intercal}Q_{T}\left(x_{T}-H\E[x_{T}|\mathcal{F}_{T}^0]\right)\bigg\},\label{MTCost} 
            \end{align} 
where the state process $x\in \mathcal{C}_{\mathcal{F}}\left([0,T],\mathbb{R}^{n}\right)$ satisfies 
\begin{equation}
    dx_{t} = [A_{t}x_{t}+B_{t}u_{t}+ F_{t}\E[x_{t}|\mathcal{F}_{t}^0] + b_{t}]dt+D_{t} dW_{t} + D_{t}^{0} dW_{t}^{0} ,  \qquad x_0 = \xi \in  L^{2}_{\mathcal{F}_0}\left(\Omega,\mathbb{R}^{n}\right). \label{MTIstate}
\end{equation}
All coefficients in the above equations are stochastic and defined as follows: $A, F\in L_{\mathcal{F}^{0}}^{\infty}\left( [0,T],\mathbb{R}^{n\times n}\right)$, $B,S\in L_{\mathcal{F}^{0}}^{\infty}\left( [0,T],\mathbb{R}^{n\times d}\right)$, $b,D,D^{0}\in L_{\mathcal{F}^{0}}^{2}\left([0,T],\mathbb{R}^{n}\right)$, $H\in L_{\mathcal{F}^{0}_{0}}^{\infty}\left(\mathbb{R}^{n\times n}\right)$, $Q\in L_{\mathcal{F}^{0}}^{\infty}\left( [0,T],\mathbb{S}^{n}\right)$, $\zeta\in L_{\mathcal{F}^{0}}^{2}\left([0,T],\mathbb{R}^{n}\right)$, $\varpi\in L_{\mathcal{F}^{0}}^{2}\left([0,T],\mathbb{R}^{d}\right)$ and $R\in L_{\mathcal{F}^{0}}^{\infty}\left( [0,T],\mathbb{S}^{d}\right)$. 
\end{probMFT}      

\begin{definition}[Admissible Pair] For an admissible control process \( u \in \mathcal{U} \), the pair \( (x, u) \), where \( x \) satisfies \eqref{MTIstate} with \( u \) as the control, is called an \textit{admissible pair} for the \hyperref[MTProb]{\color{black}\thnameref{MTProb}}.   
\end{definition}
\begin{definition}[Solvability and Optimal Pair]
     \hyperref[MTProb]{\color{black}\thnameref{MTProb}} is said to be (uniquely) solvable if there exists a (unique) minimizer of \( J(\cdot) \); that is, a (unique) admissible control \( u^{*} \in \mathcal{U} \) such that \( J(u^{*}) \leq J(u) \) for all \( u \in \mathcal{U} \). The corresponding state \( x^{*} \), which satisfies \eqref{MTIstate} under the control \( u^{*} \), is called the \emph{optimal trajectory}, and the pair \( (x^{*}, u^{*}) \) is referred to as the \emph{optimal pair} for the \hyperref[MTProb]{\color{black}\thnameref{MTProb}}.
\end{definition}
 The dependence of the model in \hyperref[MTProb]{\color{black}\thnameref{MTProb}} on the conditional expectation of the state, i.e., \( \E[x_t | \mathcal{F}_t^0] \), classifies it as a conditional Mckean-Vlasov (or equivalently, mean-field type or mean-field) optimal control problem. For instance, this model is motivated by the asymptotic formulation of systemic risk in the interbank framework of \cite[pp.~22--24]{Pham:RandCoef2016}, where the dynamics of banks' log-monetary reserves are driven by both idiosyncratic and common noise processes. In this setting, each bank regulates its borrowing and lending rates with a central bank, and a cooperative equilibrium is established from the perspective of a social planner who optimizes all banks' strategies to minimize total systemic cost. In particular, our model generalizes these control strategies by allowing the admissible borrowing and lending rates to be adapted to both noise processes instead of to only common noise.
 
 Similar to other optimal control problems, this problem can be addressed through three main steps: (i) establishing the well-posedness of the state SDE \eqref{MTIstate} and the cost functional \eqref{MTCost}, (ii) proving the solvability of \hyperref[MTProb]{\color{black}\thnameref{MTProb}}, and (iii) characterizing the corresponding optimal control.  

Mckean-Vlasov optimal control problems have been extensively studied in the literature. For any \( \xi \in L_{\mathcal{F}_0^0}^2(\Omega, \mathbb{R}^n) \) and \( u \in \mathcal{U} \), the SDE \eqref{MTIstate} admits a unique solution \( x \in \mathcal{C}_{\mathcal{F}}([0,T], \mathbb{R}^n) \), and the cost functional \( J(u) \) defined by \eqref{MTCost} is well-posed \citep{Pham:RandCoef2016}. Moreover, the solvability of such problems and the characterization of the optimal control have been established using both the stochastic maximum principle and dynamic programming approaches.  

However, these methods often require certain extensions that depend on additional, and sometimes restrictive, assumptions about the model. In particular, the extended maximum principle leads to a Hamiltonian system of Mckean-Vlasov type involving the conditional expectation of the adjoint process with respect to the filtration \( \mathcal{F}^0 \), which may introduce additional technical challenges \citep{Yang2022}. The extended dynamic programming principle, on the other hand, requires admissible control processes to be adapted only to \( \mathcal{F}^0 \), and not to \( \mathcal{F} \) \citep{Pham:RandCoef2016}.

In the following section, we propose a decomposition method that may allow the use of standard techniques without imposing additional conditions on the measurability of  admissible control inputs or requiring a Mckean-Vlasov Hamiltonian system.

\section{Decomposition Approach}\label{Sec3}
We decompose \hyperref[MTProb]{\color{black}\thnameref{MTProb}} into two auxiliary problems: (1) a standard stochastic optimal control problem, and (2) a standard stochastic optimal control problem with a constrained admissible control set. These auxiliary problems may be solved using classical methods. As a result, classical methods can  be also used to address  \hyperref[MTProb]{\color{black}\thnameref{MTProb}}, which highlights the main advantage of the proposed decomposition method.

 We begin the analysis of \hyperref[MTProb]{\color{black}\thnameref{MTProb}} using the proposed decomposition method by defining the two auxiliary stochastic optimal control problems, followed by an examination of their well-posedness, solvability, and optimal control characterization, as well as their relation to those of \hyperref[MTProb]{\color{black}\thnameref{MTProb}} in the subsequent subsections.
        \begin{prob1}[First Auxiliary Problem]\thlabel{DecouplProb1}
            
            Define an admissible control set as $\bar{\mathcal{U}}: =  L_{\mathcal{F}^{0}}^{2}\left( [0,T],\mathbb{R}^{d}\right)$. Find the optimal control process $v^{*}\in\bar{\mathcal{U}}$ that minimizes the cost functional $\bar{J}\left(v\right)$ given by 
                \begin{equation}\label{BarCost}
                    \begin{split}
                        \bar{J}\left(v\right) &:= \frac{1}{2}\E\left\{\int_{0}^{T}\left[y^{\intercal}_{t}\bar{Q}_{t}y_{t} + 2y^{\intercal}_{t}\bar{S}_{t}v_{t} + 2\bar{\zeta}^{\intercal}y_{t} + 2\varpi_{t}^{\intercal}v_{t} + v_{t}^{\intercal}R_{t}v_{t}\right]dt + y^{\intercal}_{T}\bar{Q}_{T}y_{T}\right\},
                    \end{split}
                \end{equation} where $\bar{S}_{t}:=\left(I_{n} - H\right)^{\intercal}S_{t},\, \bar{\zeta}_t:= \left(I_{n} - H\right)^{\intercal}\zeta_t\mbox{ and } \bar{Q}_{t}:=\left(I_{n} - H\right)^{\intercal}Q_{t}\left(I_{n} - H\right)$, and the state process $y \in \mathcal{C}_{\mathcal{F}^{0}}\left([0,T],\mathbb{R}^{n}\right)$ satisfies
            \begin{equation}\label{BarState}
                    dy_{t} = [(A_{t} + F_{t})y_{t} + B_{t}v_{t} + b_{t}]dt+ D_{t}^{0} dW_{t}^0, \quad 
                    y_0 =\bar{\xi}\in \mathbb{R}^{n}.
            \end{equation} 
            All other coefficients in the above equations are defined as in \hyperref[MTProb]{\color{black}\thnameref{MTProb}}.
        \end{prob1}

        \begin{prob2}[Second Auxiliary Problem]\thlabel{DecouplProb2}
             Define an admissible control set as $\breve{\mathcal{U}}: = \Big\{ \alpha \in L_{\mathcal{F}}^{2}\left( [0,T],\mathbb{R}^{d}\right):\allowdisplaybreaks \E\left[\alpha_{t} \big| \mathcal{F}^{0}_{t}\right]=0,\,dt\otimes d\mathbb{P}-a.s. \Big\}$. Find the optimal control process $\alpha^\ast \in \breve{\mathcal{U}}$ that minimizes the cost functional $\breve{J}\left(\alpha\right)$ given by 
            \begin{equation}\label{BrevCost}
                \breve{J}\left(\alpha\right):= \frac{1}{2}\E\left\{\int_{0}^{T}\left[{z_{t}}^{\intercal}Q_{t}z_{t}+2{z_{t}}^{\intercal}S_{t}\alpha_{t} + {\alpha_{t}}^{\intercal}R_{t}\alpha_{t}\right]dt + {z_{T}}^{\intercal}Q_{T}z_{T}\right\},
            \end{equation} 
            where the state process $z \in \mathcal{C}_{\mathcal{F}}\left([0,T],\mathbb{R}^{n}\right)$ satisfies
            \begin{equation}\label{BrevState}
                    dz_{t} = \left[A_{t}z_{t} + B_{t}\alpha_{t}\right]dt + D_{t} dW_{t}, \quad
                    z_0 = \breve{\xi}\in \left\{\breve{\xi} \in L^{2}_{\mathcal{F}_0}\left(\Omega,\mathbb{R}^{n}\right): \E\big[\breve{\xi}\big]=0\right\}.
               \end{equation}
          All the coefficients in the above equations are defined as in \hyperref[MTProb]{\color{black}\thnameref{MTProb}}.    
        \end{prob2}
    \begin{remark}
    From \eqref{BrevCost}-\eqref{BrevState}, \hyperref[DecouplProb2]{\color{black}\thnameref{DecouplProb2}} is a stochastic optimal control problem involving dynamics driven by the Wiener process \(W_t\), 
with coefficients adapted to the filtration generated by an exogenous Wiener process \(W^0_t\). 
Although distinct from the setting of this paper, similar dynamics arise in limiting LQ mean field games with a major agent (see, e.g., \cite{huang2010large, firoozi2020convex, carmona2016probabilistic,nourian2013epsilon,firoozi2022LQG,huang2020linear,carmona2017alternative}), 
where the major agent’s state and the mean field, both adapted to the filtration generated by \(W^0_t\), 
serve as exogenous stochastic coefficients in the dynamics of the representative minor agent driven by \(W_t\) 
(see, e.g., \cite[Section 3]{carmona2017alternative} and \cite[Section 5]{firoozi2020convex}).
    \end{remark}
    \begin{remark}\label{Remark1} For any
       $v \in \bar{\mathcal{U}}$  and any $\alpha \in \breve{\mathcal{U}}$, we have $v,\alpha\in\mathcal{U}$ and
       \[
            \langle v, \alpha \rangle_{\mathcal{U}} = \mathbb{E} \int_0^T v_t^\intercal \alpha_t dt = \mathbb{E} \int_0^T \E\left[v_t^\intercal \alpha_t| \mathcal{F}^{0}_{t}\right] dt = \mathbb{E} \int_0^T v_t^\intercal\E\left[ \alpha_t| \mathcal{F}^{0}_{t}\right] dt = 0,
       \] which indicates the orthogonality of the control sets $\bar{\mathcal{U}}$ and $\breve{\mathcal{U}}$ within $\mathcal{U}$.
    \end{remark}
    \subsection{Well-posedness}
        Based on existing results related to stochastic optimal control problems with random coefficients \citep{Bismut1976,Tang2003,Sun2021},
        for any $\bar{\xi} \in \mathbb{R}^{n}$ and $v\in \bar{\mathcal{U}}$ (respectively, $\breve{\xi} \in \big\{\breve{\xi} \in L^{2}_{\mathcal{F}_0}\left(\Omega,\mathbb{R}^{n}\right): \E\big[\breve{\xi}\big]=0\big\} $ and $\alpha\in \breve{\mathcal{U}}$), the SDE \eqref{BarState} \big(respectively, \eqref{BrevState}\big) admits a unique solution 
        $y\in \mathcal{C}_{\mathcal{F}^{0}}\left([0,T],\mathbb{R}^{n}\right)$ (respectively, $z\in\{z\in\mathcal{C}_{\mathcal{F}}\left([0,T],\mathbb{R}^{n}\right)$: $\E\left[z_{t} | \mathcal{F}^{0}_{t}\right]=0\,dt\otimes d\mathbb{P}-a.s.\}$). 
        
    In the following, we establish the relationship between the admissible pairs of \hyperref[MTProb]{\color{black}\thnameref{MTProb}} and those of \hyperref[DecouplProb1]{\color{black}\thnameref{DecouplProb1}} and \hyperref[DecouplProb2]{\color{black}\thnameref{DecouplProb2}}. For convenience, for any admissible pair $\left(x,u\right)$  of \hyperref[MTProb]{\color{black}\thnameref{MTProb}}, we denote
    \begin{align}\label{Decomp-state-control}
    \begin{array}{cc}
       \;\;\bar{x}_{t}:= \E[x_{t}|\mathcal{F}_{t}^0],  &\;\;\bar{u}_{t} := \E[u_{t}|\mathcal{F}_{t}^0], \\ 
        \breve{x}_{t}:= x_{t} - \bar{x}_{t},  &\breve{u}_{t}:= u_{t} - \bar{u}_{t}.
    \end{array}
    \end{align}
     \begin{assumption}(Initial Conditions)\label{ass:init-cond}
     For any given  initial condition $\xi \in L^{2}_{\mathcal{F}_0^0}\left(\Omega,\mathbb{R}^{n}\right)$ for \eqref{MTIstate}, the initial conditions of \eqref{BarState} and \eqref{BrevState} are, respectively, set to
    $
    \bar{\xi} = \E[\xi]$ and $ \breve{\xi} = \xi - \E[\xi]$.    
  \end{assumption}
   
   \begin{proposition}(Admissibility)\label{StateSplit}
        Suppose that \Cref{ass:init-cond} holds. Then, the following statements hold:
        \begin{itemize}
            \item[(a)]  If  $\left(y,v\right)\in \mathcal{C}_{\mathcal{F}^{0}}\left([0,T],\mathbb{R}^{n}\right)\times\bar{\mathcal{U}}$ is  an admissible pair for \hyperref[DecouplProb1]{\color{black}\thnameref{DecouplProb1}} and $\left(z,\alpha\right)\in \mathcal{C}_{\mathcal{F}}\left([0,T],\mathbb{R}^{n}\right)\times\breve{\mathcal{U}}$ an admissible pair  for \hyperref[DecouplProb2]{\color{black}\thnameref{DecouplProb2}},  then the pair $\left(y +z ,v + \alpha\right)\in \mathcal{C}_{\mathcal{F}}\left([0,T],\mathbb{R}^{n}\right)\times\mathcal{U}$ is admissible for \hyperref[MTProb]{\color{black}\thnameref{MTProb}}.
            \item[(b)] If  $\left(x,u\right)\in \mathcal{C}_{\mathcal{F}}\left([0,T],\mathbb{R}^{n}\right)\times\mathcal{U}$ is  an  admissible pair for \hyperref[MTProb]{\color{black}\thnameref{MTProb}}, then $\left(\bar{x},\bar{u}\right)$ is admissible for \hyperref[DecouplProb1]{\color{black}\thnameref{DecouplProb1}} and $\left(\breve{x},\breve{u}\right)$  admissible for \hyperref[DecouplProb2]{\color{black}\thnameref{DecouplProb2}}.
        \end{itemize}
       
    \end{proposition}
    \begin{proof}
         We begin with statement (a). If $\left(y,v\right)\in \mathcal{C}_{\mathcal{F}^{0}}\left([0,T],\mathbb{R}^{n}\right)\times\bar{\mathcal{U}}$ and $\left(z,\alpha\right)\in \mathcal{C}_{\mathcal{F}}\left([0,T],\mathbb{R}^{n}\right)\times\breve{\mathcal{U}}$ are admissible for  \hyperref[DecouplProb1]{\color{black}\thnameref{DecouplProb1}} and \hyperref[DecouplProb2]{\color{black}\thnameref{DecouplProb2}}, respectively, then it is clear that $y +z\in \mathcal{C}_{\mathcal{F}}\left([0,T],\mathbb{R}^{n}\right)$ and $v + \alpha\in\mathcal{U}$.
         
         Furthermore, from \eqref{BarState} and \eqref{BrevState}, $y + z$ satisfies
        \begin{equation}\label{dynamics-y+z}
                d(y_{t} + z_{t}) = [A_{t}(y_{t} + z_{t}) + B_{t}(v_{t} + \alpha_{t}) + F_{t}y_{t} + b_{t}]dt+ D_{t} dW_{t} + D^{0}_{t} dW_{t}^{0}, \quad
                y_{0} + z_{0} = \xi.
        \end{equation} 
        Since $\E[y_{t} + z_{t} | \mathcal{F}_{t}^0] = y_{t}$ for every $t\in[0,T]$, it follows that \eqref{dynamics-y+z} coincides with \eqref{MTIstate}, implying that $y + z$ satisfies \eqref{MTIstate}. Hence, $\left(y +z ,v + \alpha\right)$ is an admissible pair for \hyperref[MTProb]{\color{black}\thnameref{MTProb}}.

        Now, we verify the statement (b). Let $(x,u)\in \mathcal{C}_{\mathcal{F}}\left([0,T],\mathbb{R}^{n}\right)\times\mathcal{U}$ be admissible for \hyperref[MTProb]{\color{black}\thnameref{MTProb}}. Note that from \eqref{Decomp-state-control} $\bar{u}\in \bar{\mathcal{U}}$ and $\breve{u}\in \breve{\mathcal{U}}$, and from \eqref{MTIstate} we obtain 
            \begin{equation}
                \begin{array}{cc}
                    d\bar{x}_{t} = \left[(A_{t} + F_{t})\bar{x}_{t} + B_{t}\bar{u}_{t} + b_{t}\right]dt+ D_{t}^{0} dW_{t}^0,  & \bar{x}_{0} =\E[\xi], \\
                    d\breve{x}_{t} = \left[A_{t}\breve{x}_{t} + B_{t}\breve{u}_{t}\right]dt + D_{t} dW_{t},\qquad \qquad \,\quad &
                    \qquad\breve{x}_{0} = \xi - \E[\xi].
                \end{array}
            \end{equation} 
            That is, for $\bar{\xi} = \E[\xi]$, $\breve{\xi} = \xi - \E[\xi]$, $v=\bar{u}$ and $\alpha=\breve{u}$, the processes $\bar{x}$ and $\breve{x}$ satisfy \eqref{BarState} and \eqref{BrevState}, respectively. Therefore, for any admissible pair $(x,v)$ for \hyperref[MTProb]{\color{black}\thnameref{MTProb}}, the pairs $(\bar{x},\bar{u})$ and $(\breve{x},\breve{u})$ are admissible for \hyperref[DecouplProb1]{\color{black}\thnameref{DecouplProb1}} and \hyperref[DecouplProb2]{\color{black}\thnameref{DecouplProb2}}, respectively.
    \end{proof}
    
\subsection{Solvability}
In this section, we demonstrate that the solvability of \hyperref[MTProb]{\color{black}\thnameref{MTProb}} is equivalent to that of \hyperref[DecouplProb1]{\color{black}\thnameref{DecouplProb1}} and \hyperref[DecouplProb2]{\color{black}\thnameref{DecouplProb2}}, which constitutes a key result of our approach. We begin with the following lemma.

\begin{lemma}[Relevant Decompositions]\label{lemma:cost-term-decomp}
    Suppose that \Cref{ass:init-cond} holds. Then, for any admissible  pair $\left(x,u\right)$ for \hyperref[MTProb]{\color{black}\thnameref{MTProb}} and any $t\in [0,T]$, the following properties hold: 
        
            \begin{enumerate}[(i)]
                \item $\E\left[\zeta_{t}^{\intercal}\left(x_{t} - H\bar{x}_{t}\right)\right] = \E\big[\bar{\zeta_{t}}^{\intercal}\bar{x}_{t}\big]$,

                \item $\E[\varpi_{t}^{\intercal}u_{t}] = \E[\varpi_{t}^{\intercal}\bar{u}_{t}]$,

                \item $\E[u_{t}^{\intercal}R_{t}u_{t}] = \E[\breve{u}_{t}^{\intercal}R_{t}\breve{u}_{t}] + \E[\bar{u}_{t}^{\intercal}R_{t}\bar{u}_{t}]$,

                \item $\E[\left(x_{t} - H\bar{x}_{t}\right)^{\intercal}S_{t}u_{t}] = \E[\breve{x}_{t}^{\intercal}S_{t}\breve{u}_{t}] + \E[\bar{x}^{\intercal}_{t}\bar{S}_{t}\bar{u}_{t}]$,

                \item $\E[\left(x_{t} - H\bar{x}_{t}\right)^{\intercal}Q_{t}\left(x_{t} - H\bar{x}_{t}\right)] = \E[\breve{x}^{\intercal}_{t}Q_{t}\breve{x}_{t}] + \E\left[\bar{x}_{t}^{\intercal}\bar{Q}_{t}\bar{x}_{t} \right]$,
            \end{enumerate}
        where $(\bar{x},\bar{u})$ and $(\breve{x},\breve{u})$ are as given in \eqref{Decomp-state-control}. Moreover, $\bar{S}_{t}:=\left(I_{n} - H\right)^{\intercal}S_{t},\, \bar{\zeta}_t = \left(I_{n} - H\right)^{\intercal}\zeta_t\mbox{ and } \bar{Q}_{t}:=\left(I_{n} - H\right)^{\intercal}Q_{t}\left(I_{n} - H\right)$, as given in the description of \hyperref[DecouplProb1]{\color{black}\thnameref{DecouplProb1}}. 
        \end{lemma}
        \begin{proof}
        Using the smoothing property of conditional expectations, i.e. $\E\big[\E[.|\mathcal{F}^0_t]\big]=\E[.]$ for every $t\in [0,T]$ in our context, (i), (ii) and (iii) can be established. The detailed proofs of (iv) and (v) are deferred to \ref{DecomProof}.
      \end{proof}      
    
    \begin{corollary}[Cost Functional Decomposition]\label{CostSplit}
        Suppose that \Cref{ass:init-cond} holds. Then, for any admissible pair $(x,u)\in\mathcal{C}_{\mathcal{F}}\left([0,T],\mathbb{R}^{n}\right)\times\mathcal{U}$, the cost functional $J(u)$, given by \eqref{MTCost}, can be expressed as 
        \begin{equation}
            J\left(u\right) = \bar{J}\left(\bar{u}\right) + \breve{J}\left(\breve{u}\right),
        \end{equation}
        where $\bar{J}(.)$, $\breve{J}(.)$, and $(\bar{x},\bar{u})$ and $(\breve{x},\breve{u})$ are, respectively, given by \eqref{BarCost}, \eqref{BrevCost}, and \eqref{Decomp-state-control}. 
    \end{corollary}
    \begin{proof} 
        
        From \Cref{StateSplit} and the properties (i)-(v), for any admissible pair $(x,u)$ for \hyperref[MTProb]{\color{black}\thnameref{MTProb}}, we have
              \begin{align*}
                    J\left(u\right) &= \frac{1}{2}\E\bigg\{\int_{0}^{T}\left[(x_{t}-H\bar{x}_{t})^{\intercal}Q_{t}(x_{t}-H\bar{x}_{t})+2(x_{t}-H\bar{x}_{t})^{\intercal}S_{t}u_{t} + 2\zeta^{\intercal}_{t}\left(x_{t} - H\bar{x}_{t}\right) + 2\varpi_{t}^{\intercal}u_{t} + u^{\intercal}_{t}R_{t}u_{t}\right]dt\nonumber\allowdisplaybreaks\\
                    &\qquad \quad + (x_{T}-H\bar{x}_{T})^{\intercal}Q_{T}(x_{T}-H\bar{x}_{T})\bigg\}\nonumber\allowdisplaybreaks\\
                    & = \frac{1}{2}\E\bigg\{\int_{0}^{T}\left[\bar{x}^{\intercal}_{t}\bar{Q}_{t}\bar{x}_{t} + 2\bar{x}^{\intercal}_{t}\bar{S}_{t}\bar{u}_{t} + 2\bar{\zeta}^{\intercal}\bar{x}_{t} + 2\varpi_{t}^{\intercal}\bar{u}_{t} + \bar{u}_{t}^{\intercal}R_{t}\bar{u}_{t}\right]dt + \bar{x}^{\intercal}_{T}\bar{Q}_{T}\bar{x}_{T}\bigg\}\nonumber\allowdisplaybreaks\\
                    &\quad + \frac{1}{2}\E\bigg\{\int_{0}^{T}\left[{\breve{x}_{t}}^{\intercal}Q_{t}\breve{x}_{t}+2{\breve{x}_{t}}^{\intercal}S_{t}\breve{u}_{t} + {\breve{u}_{t}}^{\intercal}R_{t}\breve{u}_{t}\right]dt + {\breve{x}_{T}}^{\intercal}Q_{T}\breve{x}_{T}\bigg\}\nonumber\allowdisplaybreaks\\
                    &= \bar{J}\left(\bar{u}\right) + \breve{J}\left(\breve{u}\right).
                \end{align*} 
    \end{proof}  
    The following theorem contains the main result of the paper and relies directly on \Cref{StateSplit}, \Cref{lemma:cost-term-decomp} and \Cref{CostSplit}.
    \begin{theorem}(Solvability)\label{KeyTheo} 
        Suppose that \Cref{ass:init-cond} holds. Then, \hyperref[MTProb]{\color{black}\thnameref{MTProb}} is solvable if and only if both \hyperref[DecouplProb1]{\color{black}\thnameref{DecouplProb1}} and \hyperref[DecouplProb2]{\color{black}\thnameref{DecouplProb2}} are solvable. Furthermore, 
    \begin{itemize}
        \item[(a)] If the pair $(y^{*},v^{*})$ is optimal for \hyperref[DecouplProb1]{\color{black}\thnameref{DecouplProb1}} and the pair $(z^{*},\alpha^{*})$ is optimal for \hyperref[DecouplProb2]{\color{black}\thnameref{DecouplProb2}}, then the pair $(x,u):= (y^{*} + z^{*},v^{*} + \alpha^{*})$ is optimal for \hyperref[MTProb]{\color{black}\thnameref{MTProb}};

        \item[(b)] If  the pair $(x^{*},u^{*})$ is optimal for \hyperref[MTProb]{\color{black}\thnameref{MTProb}}, then the pair $(y_{t},v_{t}):=\left(\E[x^{*}_{t}|\mathcal{F}_{t}^0],\E[u^{*}_{t}|\mathcal{F}_{t}^0]\right)$ is optimal for \hyperref[DecouplProb1]{\color{black}\thnameref{DecouplProb1}}, and the pair $(z_{t},\alpha_{t}):=\left(x^{*}_{t} - \E[x^{*}_{t}|\mathcal{F}_{t}^0],u^{*}_{t} - \E[u^{*}_{t}|\mathcal{F}_{t}^0]\right)$ is optimal for \hyperref[DecouplProb2]{\color{black}\thnameref{DecouplProb2}}.
    \end{itemize}
    \end{theorem}
    \begin{proof} 
           First we show that if \hyperref[DecouplProb1]{\color{black}\thnameref{DecouplProb1}} and \hyperref[DecouplProb2]{\color{black}\thnameref{DecouplProb2}} are solvable, then \hyperref[MTProb]{\color{black}\thnameref{MTProb}} is also solvable and statement (a) holds. Suppose that both \hyperref[DecouplProb1]{\color{black}\thnameref{DecouplProb1}} and \hyperref[DecouplProb2]{\color{black}\thnameref{DecouplProb2}} are solvable; that is, there exist optimal pairs $(y^{*},v^{*})$ and $(z^{*},\alpha^{*})$ such that 
                \begin{equation}
                    \begin{split} \label{FirstSplitIneq}
                        \bar{J}\left(v^{*}\right)\leq \bar{J}\left(v\right),\quad\forall v \in \bar{\mathcal{U}} \quad \text{and} \quad                        \breve{J}\left(\alpha^{*}\right) \leq \breve{J}\left(\alpha\right),\quad \forall \alpha\in \breve{\mathcal{U}}.
                    \end{split}
                \end{equation} 
        According to \Cref{StateSplit}, for any admissible pair $\left(x,u\right)$ of \hyperref[MTProb]{\color{black}\thnameref{MTProb}} the corresponding pairs $\left(\bar{x},\bar{u}\right)$ and $\left(\breve{x},\breve{u}\right)$, as defined in \eqref{Decomp-state-control}, are admissible for \hyperref[DecouplProb1]{\color{black}\thnameref{DecouplProb1}} and \hyperref[DecouplProb2]{\color{black}\thnameref{DecouplProb2}}, respectively. Hence, from \eqref{FirstSplitIneq}, we have
                \begin{equation}
                    \bar{J}\left(v^{*}\right)\leq \bar{J}\left(\bar{u}\right) \quad \mbox{ and } \quad \breve{J}\left(\alpha^{*}\right) \leq \breve{J}\left(\breve{u}\right).\label{mid-step}
                \end{equation}
        From \Cref{CostSplit} and \eqref{mid-step}, for any admissible pair $\left(x,u\right)$ of \hyperref[MTProb]{\color{black}\thnameref{MTProb}}, we have  
            \begin{equation}
                J\left(u\right) = \bar{J}\left(\bar{u}\right) + \breve{J}\left(\breve{u}\right) \geq \bar{J}\left(v^{*}\right) + \breve{J}\left(\alpha^{*}\right) = J\left(v^{*} + \alpha^{*}\right),\quad \forall u\in\mathcal{U},
            \end{equation}
            where the last equality holds because $\E[v^*_t+\alpha^*_t|\mathcal{F}^0_t]=v^*_t$ and $\E[y^*_t+z^*_t|\mathcal{F}^0_t]=y^*_t$ for every $t\in[0,T]$ under \Cref{ass:init-cond}.
            Hence, $(y^{*} + z^{*},v^{*} + \alpha^{*})$ is an optimal pair for $J(\cdot)$, and therefore, \hyperref[MTProb]{\color{black}\thnameref{MTProb}} is solvable.

            We now show that if \hyperref[MTProb]{\color{black}\thnameref{MTProb}} is solvable, then both \hyperref[DecouplProb1]{\color{black}\thnameref{DecouplProb1}} and \hyperref[DecouplProb2]{\color{black}\thnameref{DecouplProb2}} are solvable and statement (b) holds. Suppose that \hyperref[MTProb]{\color{black}\thnameref{MTProb}} is solvable; that is, there exists an optimal pair $(x^{*},u^{*})$ such that
               \begin{equation}
                    J\left(u^{*}\right)\leq J\left(u\right),\quad \forall u \in \mathcal{U}.
                \end{equation}
            According to \Cref{StateSplit}, for any admissible pair $(y,v)$ for \hyperref[DecouplProb1]{\color{black}\thnameref{DecouplProb1}}, $(y+\breve{x}^{*},v+\breve{u}^{*})$ is an admissible pair for \hyperref[MTProb]{\color{black}\thnameref{MTProb}}. Hence, from \Cref{CostSplit}, we have  
            \begin{equation}
                 J\left(v + \breve{u}^{*}\right) = \bar{J}\left(v\right) + \breve{J}\left(\breve{u}^{*}\right),
            \end{equation} since $\E[v_t+\breve{u}^*_t|\mathcal{F}^0_t]=v_t$ and $\E[y_t+\breve{x}^*_t|\mathcal{F}^0_t]=y_t$ for every $t\in[0,T]$ under \Cref{ass:init-cond}.
            Moreover, since $u^{*} = \bar{u}^{*} + \breve{u}^{*}$, we have
            \begin{equation}
                \begin{split}
                    \bar{J}\left(\bar{u}^{*}\right) = J\left(\bar{u}^{*} + \breve{u}^{*}\right) - \breve{J}\left(\breve{u}^{*}\right)\leq J\left(v + \breve{u}^{*}\right) - \breve{J}\left(\breve{u}^{*}\right) = \bar{J}\left(v\right),\quad\forall v\in\bar{\mathcal{U}}.
                \end{split}
            \end{equation} Consequently, the pair $(\bar{x}^{*},\bar{u}^{*})$ is a optimal pair for \hyperref[DecouplProb1]{\color{black}\thnameref{DecouplProb1}} and making it solvable. Similarly, it can be shown that the pair $(\breve{x}^{*},\breve{u}^{*})$ is optimal for \hyperref[DecouplProb2]{\color{black}\thnameref{DecouplProb2}} and making it solvable.
        \end{proof}
     We now discuss uniform convexity, which provides a sufficient condition for the unique solvability of the problems under consideration.
    \begin{assumption}\label{A2}
        $R_{t}\geq \delta I_{d}$, with $\delta > 0$ and $I_{d}$ being the identity matrix on $\mathbb{R}^{d\times d}$, $  Q_{t} - S_{t}R^{-1}_{t}S^{\intercal}_{t}\geq 0, \,dt\otimes d\mathbb{P}-a.s$ and $Q_{T}\geq 0 \,d\mathbb{P}-a.s$.
    \end{assumption}

\begin{proposition}(Uniform Convexity)\label{UnifConvChara}
    Under \Cref{A2}, the cost functionals $J(u)$, $\bar{J}(v)$ and $\breve{J}(\alpha)$, respectively associated with \hyperref[MTProb]{\color{black}\thnameref{MTProb}}, \hyperref[DecouplProb1]{\color{black}\thnameref{DecouplProb1}} and \hyperref[DecouplProb2]{\color{black}\thnameref{DecouplProb2}}, are uniformly convex.
\end{proposition}
\begin{proof}
    From \cite[Theorem $4.2$(iii)]{XYZ2012} and \cite[Corollary $3.5$(ii)]{Sun2021}, we have 
     \begin{enumerate}
        \item The cost functional $\bar{J}(v)$ of \hyperref[DecouplProb1]{\color{black}\thnameref{DecouplProb1}} is uniformly convex, if there exists $\delta_{1}>0$ such that
            \begin{equation}\label{UnifCovexCondBar}
                \E\left\{\int_{0}^{T}\left[{y^{0}_{t}}^{\intercal}\bar{Q}_{t}y^{0}_{t} + 2{y^{0}_{t}}^{\intercal}\bar{S}_{t}v_{t} + {v_{t}}^{\intercal}R_{t}v_{t}\right]dt + {y^{0}_{T}}^{\intercal}\bar{Q}_{T}y^{0}_{T}\right\} \geq \delta_{1}\E\int_{0}^{T}v^{\intercal}_{t} v_{t} dt,\quad\forall v\in\bar{\mathcal{U}},
            \end{equation} where $y^{0}_t$ satisfies 
        \begin{equation}\label{ZeroBarFBSDE}
                    dy^{0}_{t} = [(A_{t} + F_{t})y^{0}_{t} + B_{t}v_{t}]dt, \quad y^{0}_{0} = 0.
        \end{equation} 

        \item The cost functional $\breve{J}(\alpha)$ of \hyperref[DecouplProb2]{\color{black}\thnameref{DecouplProb2}} is uniformly convex, if  there exists $\delta_{2}>0$ such that
            \begin{equation}\label{UnifCovexCondBrev}
                \E\left\{\int_{0}^{T}\left[{z^{0}_{t}}^{\intercal}Q_{t}z^{0}_{t}+2{z^{0}_{t}}^{\intercal}S_{t}\alpha_{t} + {\alpha_{t}}^{\intercal}R_{t}\alpha_{t}\right]dt + {z^{0}_{T}}^{\intercal}Q_{T}z^{0}_{T}\right\} \geq \delta_{2}\E\int_{0}^{T}\alpha^{\intercal}_{t} \alpha_{t} dt,\quad\forall \alpha\in\breve{\mathcal{U}},
            \end{equation} where $z^{0}_t$ satisfies 
        \begin{equation}\label{ZeroBrevFBSDE}
                dz^{0}_{t} = [A_{t}z^{0}_{t} + B_{t}\alpha_{t}]dt, \quad z^{0}_{0} = 0.
        \end{equation}
    \end{enumerate}
Moreover, following the definition in \cite[pp.~$1106$]{Sun2017}, the cost functional $J(u)$ of \hyperref[MTProb]{\color{black}\thnameref{MTProb}} is uniformly convex, if there exists $\delta_{MFT}>0$ such that
       \begin{equation}\label{UnifCovexCond}
                \begin{split}
                    &\E\bigg\{\int_{0}^{T}\left[{\left(x^{0}_{t} - H\E[x^{0}_{t}\big|\mathcal{F}^{0}_{t}]\right)}^{\intercal}Q_{t}\left(x^{0}_{t} - H\E[x^{0}_{t}\big|\mathcal{F}^{0}_{t}]\right) + 2{\left(x^{0}_{t} - H\E[x^{0}_{t}\big|\mathcal{F}^{0}_{t}]\right)}^{\intercal}S_{t}u_{t} + {u_{t}}^{\intercal}R_{t}u_{t}\right]dt\\ 
                    &\quad +  {\left(x^{0}_{T} - H\E[x^{0}_{T}\big|\mathcal{F}^{0}_{T}]\right)}^{\intercal}Q_{T}\left(x^{0}_{T} - H\E[x^{0}_{T}\big|\mathcal{F}^{0}_{T}]\right)\bigg\}\geq \delta_{MFT}\E\int_{0}^{T}u^{\intercal}_{t} u_{t} dt,\quad \forall u\in\mathcal{U},
                \end{split}  
            \end{equation} 
            where $x^{0}_t$ satisfies
        \begin{equation}\label{ZeroFBSDE}
                    dx^{0}_{t} = \left[A_{t}x^{0}_{t}+B_{t}u_{t}+ F_{t}\E[x^{0}_{t}|\mathcal{F}_{t}^0]\right]dt, \quad
                    x^{0}_0 = 0.
        \end{equation}
     Using \Cref{Remark1} and the cost functional decomposition given in \Cref{CostSplit} for the case where $b=0$ and $D^{0}=D=0$, one can readily verify that \eqref{UnifCovexCondBar} and \eqref{UnifCovexCondBrev} together imply \eqref{UnifCovexCond}, with $\delta_{MFT}=\min\{\delta_{1},\delta_{2}\}$. Therefore, the uniform convexity of $\bar{J}(v)$ and $\breve{J}(\alpha)$ leads to the uniform convexity of $J(u)$. In the following, we establish the uniform convexity of $\bar{J}(v)$ and omit the proof for $\breve{J}(\alpha)$, as it follows by similar arguments.  
    
    Under \Cref{A2}, for any process $v\in\bar{\mathcal{U}}$ and its corresponding state process $y^{0}$ satisfying \eqref{ZeroBarFBSDE}, we have
    \begin{align}
        \label{UConv1}
            \E\Big\{{y^{0}_{T}}^{\intercal}\bar{Q}_{T}y^{0}_{T}\Big\} &+ \E\bigg\{\int_{0}^{T}\left[{y^{0}_{t}}^{\intercal}\bar{Q}_{t}y^{0}_{t} + 2{y^{0}_{t}}^{\intercal}\bar{S}_{t}v_{t} + {v_{t}}^{\intercal}R_{t}v_{t}\right]dt\bigg\}\nonumber\allowdisplaybreaks\\
            &\geq \E\int_{0}^{T}\left[{y^{0}_{t}}^{\intercal}\bar{Q}_{t}y^{0}_{t} + 2{y^{0}_{t}}^{\intercal}\bar{S}_{t}v_{t} + {v_{t}}^{\intercal}R_{t}v_{t}\right]dt\nonumber\allowdisplaybreaks\\
            &= \E\int_{0}^{T}\left[{y^{0}_{t}}^{\intercal}(\bar{Q}_{t} - \bar{S}_{t}R^{-1}_{t}\bar{S}^{\intercal}_{t})y^{0}_{t} + {y^{0}_{t}}^{\intercal} \bar{S}_{t}R^{-1}_{t}\bar{S}^{\intercal}_{t}y^{0}_{t} + 2{y^{0}_{t}}^{\intercal}\bar{S}_{t}v_{t} + {v_{t}}^{\intercal}R_{t}v_{t}\right]dt\nonumber\allowdisplaybreaks\\
            &= \E\int_{0}^{T}\left[{y^{0}_{t}}^{\intercal}(\bar{Q}_{t} - \bar{S}_{t}R^{-1}_{t}\bar{S}^{\intercal}_{t})y^{0}_{t} + \left(\left(R^{-1}_{t}\right)^{\frac{1}{2}}\bar{S}^{\intercal}_{t}y^{0}_{t} + R^{\frac{1}{2}}_{t}v_{t}\right)^{\intercal}\left(\left(R^{-1}_{t}\right)^{\frac{1}{2}}\bar{S}^{\intercal}_{t}y^{0}_{t} + R^{\frac{1}{2}}_{t}v_{t}\right)\right]dt\nonumber\allowdisplaybreaks\\
            &\geq \E\int_{0}^{T}\left(R^{-1}_{t}\bar{S}^{\intercal}_{t}y^{0}_{t} + v_{t}\right)^{\intercal}R_{t}\left(R^{-1}_{t}\bar{S}^{\intercal}_{t}y^{0}_{t} + v_{t}\right)dt\nonumber\allowdisplaybreaks\\
            &\geq \delta\E\int_{0}^{T}\left(R^{-1}_{t}\bar{S}^{\intercal}_{t}y^{0}_{t} + v_{t}\right)^{\intercal}\left(R^{-1}_{t}\bar{S}^{\intercal}_{t}y^{0}_{t} + v_{t}\right)dt.
        \end{align}
    Now define the bounded linear operators $\mathcal{A}, \mathcal{A}^{-1}: \bar{\mathcal{U}} \to \bar{\mathcal{U}}$, respectively, given by 
    \begin{equation}\label{operator-def}
        \begin{split}
            (\mathcal{A}v)_{t} = R^{-1}_{t}\bar{S}^{\intercal}_{t}y^{0}_{t} + v_{t} \quad \mbox{and}\quad
            (\mathcal{A}^{-1}v)_{t} = -R^{-1}_{t}\bar{S}^{\intercal}_{t}\tilde{y}^{0}_{t} + v_{t},
        \end{split}
    \end{equation}
    where $y^{0}$ satisfies \eqref{ZeroBarFBSDE} and $\tilde{y}^{0}$ satisfies 
    \begin{equation}
            d\tilde{y}^{0}_{t} = \left[\left(A_{t} + F_{t} - B_{t}R^{-1}_{t}\bar{S}^{\intercal}_{t}\right)\tilde{y}^{0}_{t} + B_{t}v_{t}\right]dt,\quad \tilde{y}^{0}_{0} = 0.
    \end{equation} 
    We can readily verify that, for any $v\in\bar{\mathcal{U}}$,
    \begin{equation}
        \mathcal{A}^{-1}\mathcal{A}v = v,
    \end{equation}
    by proceeding as in 
    \begin{equation}
        \begin{split}
            (\mathcal{A}^{-1}\mathcal{A}v)_{t} & = -R^{-1}_{t}\bar{S}^{\intercal}_{t}\tilde{y}^{0}_{t} + \mathcal{A}v_{t}\\
            & = -R^{-1}_{t}\bar{S}^{\intercal}_{t}\tilde{y}^{0}_{t} + R^{-1}_{t}\bar{S}^{\intercal}_{t}y^{0}_{t} + v_{t}\\
            & = R^{-1}_{t}\bar{S}^{\intercal}_{t}\left(y^{0}_{t} - \tilde{y}^{0}_{t}\right) + v_{t}
        \end{split}
    \end{equation}
    where 
    \begin{equation}
        \begin{array}{cc}
            dy^{0}_{t} = [(A_{t} + F_{t})y^{0}_{t} + B_{t}v_{t}]dt,  & \quad y^{0}_{0}=0, \\
            \qquad \qquad \qquad \quad \;\; d\tilde{y}^{0}_{t} = \left[\left(A_{t} + F_{t} - B_{t}R^{-1}_{t}\bar{S}^{\intercal}_{t}\right)\tilde{y}^{0}_{t} + B_{t}(\mathcal{A}v)_{t}\right]dt, &
                    \quad \tilde{y}^{0}_{0} = 0.
        \end{array}
    \end{equation}
    Substituting $(\mathcal{A}v)_{t}$ from \eqref{operator-def} in the above equation, we observe that $y^0_t-\tilde{y}^0_t$ satisfies
    \begin{equation}
        \begin{split}
            &d\left(y^{0}_{t} - \tilde{y}^{0}_{t}\right) = \left(A_{t} + F_{t} - B_{t}R^{-1}_{t}\bar{S}^{\intercal}_{t}\right)\left(y^{0}_{t} - \tilde{y}^{0}_{t}\right)dt, \quad y^{0}_{0} - \tilde{y}^{0}_{0} = 0,
        \end{split}
    \end{equation} 
    which implies that $y^{0}_{t} - \tilde{y}^{0}_{t} = 0,\, dt\otimes d\mathbb{P}-a.s.$, and consequently 
    $(\mathcal{A}^{-1}\mathcal{A}v)_{t} = v_{t}$. 

    Using \eqref{operator-def}, we have 
    \begin{align}
    \label{UConv2}        \E\int_{0}^{T}v^{\intercal}_{t}v_{t}dt &= \E\int_{0}^{T}\left(\mathcal{A}^{-1}\mathcal{A}v\right)^{\intercal}_{t}\left(\mathcal{A}^{-1}\mathcal{A}v\right)_{t}dt\allowdisplaybreaks\\
            &\leq \Vert \mathcal{A}^{-1}\Vert^{2}_{\mathcal{L}}\,\E\int_{0}^{T}\left(\mathcal{A}v\right)^{\intercal}_{t}\left(\mathcal{A}v\right)_{t}dt\allowdisplaybreaks\\
            & = \Vert \mathcal{A}^{-1}\Vert^{2}_{\mathcal{L}}\,\E\int_{0}^{T}\left(R^{-1}_{t}\bar{S}^{\intercal}_{t}y^{0}_{t} + v_{t}\right)^{\intercal}\left(R^{-1}_{t}\bar{S}^{\intercal}_{t}y^{0}_{t} + v_{t}\right)dt,
    \end{align} 
    which results in 
    \begin{equation} \label{fin-res}       \E\int_{0}^{T}\left(R^{-1}_{t}\bar{S}^{\intercal}_{t}y^{0}_{t} + v_{t}\right)^{\intercal}\left(R^{-1}_{t}\bar{S}^{\intercal}_{t}y^{0}_{t} + v_{t}\right)dt \geq \Vert \mathcal{A}^{-1}\Vert^{-2}_{\mathcal{L}}\,\E\int_{0}^{T}v^{\intercal}_{t}v_{t}dt,  
    \end{equation}
    where $\Vert \mathcal{A}^{-1}\Vert_{\mathcal{L}}:=\displaystyle \sup\left\{\Vert \mathcal{A}^{-1}v\Vert_{L_{\mathcal{F}^{0}}^{2}}\, \big| \, v \in\mathcal{\bar{U}}\mbox{ and } \Vert v\Vert_{L_{\mathcal{F}^{0}}^{2}} = 1\right\}$ is the linear operator norm of $\mathcal{A}^{-1}$. From \eqref{operator-def} and \eqref{fin-res}, it follows that \eqref{UnifCovexCondBar} holds with $\delta_{1} = \delta \Vert \mathcal{A}^{-1}\Vert^{-2}_{\mathcal{L}}$. 
\end{proof}

Thus far, our main objective has been to demonstrate that the solvability of the considered McKean-Vlasov Control Problem, namely \hyperref[MTProb]{\color{black}\thnameref{MTProb}}, is equivalent to that of both \hyperref[DecouplProb1]{\color{black}\thnameref{DecouplProb1}} and \hyperref[DecouplProb2]{\color{black}\thnameref{DecouplProb2}}. Moreover, we showed that, under \Cref{ass:init-cond}, the sum of the optimal controls corresponding to \hyperref[DecouplProb1]{\color{black}\thnameref{DecouplProb1}} and \hyperref[DecouplProb2]{\color{black}\thnameref{DecouplProb2}} yields the optimal control for \hyperref[MTProb]{\color{black}\thnameref{MTProb}}. \hyperref[DecouplProb1]{\color{black}\thnameref{DecouplProb1}} and \hyperref[DecouplProb2]{\color{black}\thnameref{DecouplProb2}} can be solved using classical methods. In the next section, we employ a variational method to complete our analysis. Moreover, while the standard dynamic programming method can clearly be applied to \hyperref[DecouplProb1]{\color{black}\thnameref{DecouplProb1}}, in \ref{AppendixB} we show that it can also be used to address \hyperref[DecouplProb2]{\color{black}\thnameref{DecouplProb2}}.

\subsection{Characterization of Optimal Control}
We now establish necessary and sufficient conditions of optimality for \hyperref[MTProb]{\color{black}\thnameref{MTProb}} by leveraging those derived for \hyperref[DecouplProb1]{\color{black}\thnameref{DecouplProb1}} and \hyperref[DecouplProb2]{\color{black}\thnameref{DecouplProb2}}. Subsequently, we provide a complete characterization of the optimal control for \hyperref[MTProb]{\color{black}\thnameref{MTProb}} in terms of two decoupled sets of classical linear forward-backward stochastic differential equations (FBSDEs), and ultimately in feedback form. To this end, we first introduce the set 
\begin{equation*}
\mathcal{M}_{\mathcal{F}}\left([0,T],\mathbb{R}^{n}\right):=\left\{\left(\lambda,\beta,\beta^{0}\right)\in\mathcal{C}_{\mathcal{F}}\left([0,T],\mathbb{R}^{n}\right)\times L_{\mathcal{F}}^{2}\left( [0,T],\mathbb{R}^{n}\right) \times L_{\mathcal{F}}^{2}\left( [0,T],\mathbb{R}^{n}\right): \E\left[\beta^{0}_{t} \big| \mathcal{F}^{0}_{t}\right]=0,\,dt\otimes d\mathbb{P}-a.s.\right\}.
\end{equation*}

\begin{theorem}(Necessary and Sufficient Conditions of Optimality)\label{OptimPrincMFT}
        Suppose that \Cref{A2} holds. An admissible pair $(x^{*},u^{*})\in\mathcal{C}_{\mathcal{F}}\left([0,T],\mathbb{R}^{n}\right)\times\mathcal{U}$ is optimal for \hyperref[MTProb]{\color{black}\thnameref{MTProb}} if and only if 
            \begin{enumerate}[(i)]
                \item the pair  $(\bar{x}^{*}_{t},\,\bar{u}^{*}_{t}) = \left(\E[x^{*}_{t}|\mathcal{F}_{t}^0],\,\E[u^{*}_{t}|\mathcal{F}_{t}^0]\right)\in \mathcal{C}_{\mathcal{F}^{0}}\left([0,T],\mathbb{R}^{n}\right)\times\bar{\mathcal{U}}$ satisfies  
                \begin{align}   &R_{t}\bar{u}^{*}_{t} + \bar{S}_{t}^{\intercal}\bar{x}^{*}_{t} + B_{t}^{\intercal}p_{t} + \varpi_{t}=0,\quad dt\otimes d\mathbb{P}-a.s,\label{opt-control-prob1}\allowdisplaybreaks\\
                &d\bar{x}^{*}_{t} = [(A_{t} + F_{t})\bar{x}^{*}_{t} + B_{t}\bar{u}^{*}_{t} + b_{t}]dt + D^{0}_{t} dW_{t}^{0},\quad \bar{x}^{*}_{0} =\E[\xi], \label{BarFBSDE1}     
                \end{align} where $(p,q)\in L_{\mathcal{F}^{0}}^{2}\left( [0,T],\mathbb{R}^{n}\right) \times L_{\mathcal{F}^{0}}^{2}\left( [0,T],\mathbb{R}^{n}\right)$ is the adapted solution of the following BSDE 
                \begin{equation}\label{BarFBSDE2}
                        dp_{t} = -\left[(A_{t} + F_{t})^{\intercal}p_{t} + \bar{Q}_{t}\bar{x}^{*}_{t} + \bar{S}_{t}\bar{u}^{*}_{t} + \bar{\zeta}_{t}\right]dt + q_{t}dW_{t}^{0},\quad p_{T} = \bar{Q}_{T}\bar{x}^{*}_{T}, 
                \end{equation} 
                and
                \item the pair  $(\breve{x}^{*}_{t},\, \breve{u}^{*}_{t}) = \left(x^{*}_{t} - \E[x^{*}_{t}|\mathcal{F}_{t}^0],\,u^{*}_{t} - \E[u^{*}_{t}|\mathcal{F}_{t}^0]\right)\in \mathcal{C}_{\mathcal{F}}\left([0,T],\mathbb{R}^{n}\right)\times\breve{\mathcal{U}}$ satisfies  
                \begin{align}        
                &R_{t}\breve{u}^{*}_{t} + S_{t}^{\intercal}\breve{x}^{*}_{t} + B_{t}^{\intercal}\lambda_{t} = 0,\quad dt\otimes d\mathbb{P}-a.s,\label{opt-control-prob2}\allowdisplaybreaks\\
                        &d\breve{x}^{*}_{t} = \left[A_{t}\breve{x}^{*}_{t} + B_{t}\breve{u}^{*}_{t}\right]dt + D_{t} dW_{t},\quad \breve{x}^{*}_{0} =\xi - \E[\xi],\label{BrevFBSDE1}  
                \end{align} where $\left(\lambda,\beta,\beta^{0}\right)\in \mathcal{M}_{\mathcal{F}}\left([0,T],\mathbb{R}^{n}\right)$ is the adapted solution of the following BSDE
                \begin{equation}\label{BrevFBSDE2}
                        d\lambda_{t} = -\left[A^{\intercal}_{t}\lambda_{t} + Q_{t}\breve{x}^{*}_{t} + S_{t}\breve{u}^{*}_{t}\right]dt + \beta_{t}dW_{t} + \beta^{0}_{t}dW^{0}_{t}, \quad\lambda_{T} = Q_{T}\breve{x}^{*}_{T}.
                \end{equation} 
            \end{enumerate}
    \end{theorem}
\begin{proof}
    We will first characterize the optimal controls for \hyperref[DecouplProb1]{\color{black}\thnameref{DecouplProb1}} and \hyperref[DecouplProb2]{\color{black}\thnameref{DecouplProb2}}, then obtain the characterization of the optimal control for \hyperref[MTProb]{\color{black}\thnameref{MTProb}}. 
        
  We begin with \hyperref[DecouplProb2]{\color{black}\thnameref{DecouplProb2}}, which is more involved, and employ a perturbation method, following the approach in \cite{firoozi2020convex, Sun2017}, to address it. Consider a positive constant $\mu>0$ and any two admissible control inputs  $\alpha^{*},\alpha\in\breve{\mathcal{U}}$. For every $t\in [0,T]$, the state process  $z^{\alpha^{*} + \mu\alpha}$, satisfying \eqref{BrevState} under the control input  $\alpha^{*} + \mu\alpha$, is decomposed as 
  \begin{equation}
      z^{\alpha^{*} + \mu\alpha}_{t} = z_{t} + \mu z^{0}_{t}, 
  \end{equation} where $z$ denotes the state process satisfying \eqref{BrevState}  under the control input $\alpha^{*}$, and $z^{0}$ satisfies
    \begin{equation}
        dz^{0}_{t} = [A_{t}z^{0}_{t} + B_{t}\alpha_{t}]dt, \quad z^{0}_{0} = 0.
    \end{equation} Subsequently, the cost functional $\breve{J}\left(\alpha^{*} + \mu \alpha\right)$ is decomposed as  
    \begin{align}
       \label{VarBarCost}
        \breve{J}\left(\alpha^{*} + \mu \alpha\right) &= \breve{J}\left(\alpha^{*}\right) + \frac{\mu^{2}}{2}\E\left\{\int_{0}^{T}\left[{z^{0}_{t}}^{\intercal}Q_{t}z^{0}_{t}+2{z^{0}_{t}}^{\intercal}S_{t}\alpha_{t} + {\alpha_{t}}^{\intercal}R_{t}\alpha_{t}\right]dt + {z^{0}_{T}}^{\intercal}Q_{T}z^{0}_{T}\right\}\nonumber\allowdisplaybreaks\\
        &\quad + \mu\E\left\{\int_{0}^{T}\left[{z_{t}}^{\intercal}Q_{t}z^{0}_{t} + \alpha^{\intercal}_{t}S^{\intercal}_{t}z_{t} + {z^{0}_{t}}^{\intercal}S_{t}\alpha^{*}_{t} + {\alpha_{t}}^{\intercal}R_{t}\alpha^{*}_{t}\right]dt + {z^{0}_{T}}^{\intercal}Q_{T}z_{T}\right\}.
    \end{align}
    Moreover, consider the  processes $\left(\lambda,\beta,\beta^{0}\right)\in\mathcal{M}_{\mathcal{F}}\left([0,T],\mathbb{R}^{n}\right)$ satisfying     \begin{equation}\label{obtained-adjoint}
        d\lambda_{t} = -\left[A^{\intercal}_{t}\lambda_{t} + Q_{t}z_{t} + S_{t}\alpha^{*}_{t}\right]dt + \beta_{t}dW_{t} + \beta^{0}_{t}dW^{0}_{t}, \quad \lambda_{T} = Q_{T}z_{T}.
    \end{equation}
    By applying Itô's formula to ${z^0_t}^{\intercal}\lambda_t$, we obtain 
    \begin{equation}
        \begin{split}
            d\left({z^{0}_{t}}^{\intercal}\lambda_{t}\right) &= \left[A_{t}z^{0}_{t} + B_{t}\alpha_{t}\right]^{\intercal}\lambda_{t}dt - {z^{0}_{t}}^{\intercal}\left[A^{\intercal}_{t}\lambda_{t} + Q_{t}z_{t} + S_{t}\alpha^{*}_{t}\right]dt + {z^{0}_{t}}^{\intercal}\beta_{t}dW_{t} + {z^{0}_{t}}^{\intercal}\beta^{0}_{t}dW^{0}_{t},\\
            &= \left[ \alpha^{\intercal}_{t}B^{\intercal}_{t}\lambda_{t} - {z^{0}_{t}}^{\intercal}Q_{t}z_{t} - {z^{0}_{t}}^{\intercal}S_{t}\alpha^{*}_{t}\right]dt + {z^{0}_{t}}^{\intercal}\beta_{t}dW_{t} + {z^{0}_{t}}^{\intercal}\beta^{0}_{t}dW^{0}_{t}.
        \end{split} 
    \end{equation} 
    Next, by integrating both sides of the above equation over $[0, T]$ and taking expectations, we obtain 
    \begin{equation}\label{ExpZeroTErmCost}
        \E\left[{z^{0}_{T}}^{\intercal}Q_{T}z_{T}\right] = \E\int^{T}_{0}\left[ \alpha^{\intercal}_{t}B^{\intercal}_{t}\lambda_{t} - {z^{0}_{t}}^{\intercal}Q_{t}z_{t} - {z^{0}_{t}}^{\intercal}S_{t}\alpha^{*}_{t}\right]dt.
    \end{equation}
    Then, by substituting \eqref{ExpZeroTErmCost} into \eqref{VarBarCost}, we obtain
        \begin{align}
            \breve{J}\left(\alpha^{*} + \mu \alpha\right) &= \breve{J}\left(\alpha^{*}\right) + \frac{\mu^{2}}{2}\E\left\{\int_{0}^{T}\left[{z^{0}_{t}}^{\intercal}Q_{t}z^{0}_{t}+2{z^{0}_{t}}^{\intercal}S_{t}\alpha_{t} + {\alpha_{t}}^{\intercal}R_{t}\alpha_{t}\right]dt + {z^{0}_{T}}^{\intercal}Q_{T}z^{0}_{T}\right\}\nonumber\allowdisplaybreaks\\
            &\quad + \mu\E\int_{0}^{T}\alpha^{\intercal}_{t}\left[ R_{t}\alpha^{*}_{t} + S^{\intercal}_{t}z_{t} + B^{\intercal}_{t}\lambda_{t}\right]dt,\quad \forall \mu>0,\quad \forall\alpha^{*},\alpha\in\breve{\mathcal{U}}.
        \end{align}
    Consequently, under \Cref{A2}, $\alpha^{*}\in\breve{\mathcal{U}}$ is an optimal control for \hyperref[DecouplProb2]{\color{black}\thnameref{DecouplProb2}} if and only if 
            \begin{align}
                \breve{J}\left(\alpha^{*} + \mu \alpha\right) - \breve{J}\left(\alpha^{*}\right) &= \frac{\mu^{2}}{2}\E\left\{\int_{0}^{T}\left[{z^{0}_{t}}^{\intercal}Q_{t}z^{0}_{t}+2{z^{0}_{t}}^{\intercal}S_{t}\alpha_{t} + {\alpha_{t}}^{\intercal}R_{t}\alpha_{t}\right]dt + {z^{0}_{T}}^{\intercal}Q_{T}z^{0}_{T}\right\}\nonumber\allowdisplaybreaks\\
            &\quad + \mu\E\left\{\int_{0}^{T}\alpha^{\intercal}_{t}\left[ R_{t}\alpha^{*}_{t} + S^{\intercal}_{t}z_{t} + B^{\intercal}_{t}\lambda_{t}\right]dt\right\} \geq0,\quad \forall\mu>0,\quad \forall \alpha\in\breve{\mathcal{U}}.
            \end{align} 
 From  \eqref{UnifCovexCondBrev}, the above inequality holds if and only if 
        \begin{equation}\label{optim:cond:brev}
            R_{t}\alpha^{*}_{t} + S^{\intercal}_{t}z_{t} + B^{\intercal}_{t}\lambda_{t} = 0,\quad dt\otimes d\mathbb{P}-a.s.
        \end{equation} Hence, the optimal control is characterized by
        \begin{equation}\label{obtained-control}
            \alpha^{*}_{t} = -R^{-1}_{t}\left[S_{t}^{\intercal}z_{t} + B_{t}^{\intercal}\lambda_{t}\right],\quad dt\otimes d\mathbb{P}-a.s,
        \end{equation} 
        
         where $(z,\lambda,\beta,\beta^{0})\in \mathcal{C}_{\mathcal{F}}\left([0,T],\mathbb{R}^{n}\right)\times\mathcal{M}_{\mathcal{F}}\left([0,T],\mathbb{R}^{n}\right)$ is the  solution of the coupled FBSDEs given by 
    \begin{equation}\label{OptimBrevHamil:proof}
        \begin{cases}
            dz_{t} = \left[\left(A_{t} - B_{t}R^{-1}_{t}S^{\intercal}_{t}\right)z_{t} - B_{t}R^{-1}_{t}B^{\intercal}_{t}\lambda_{t}\right]dt + D_{t} dW_{t},\qquad \qquad \qquad \qquad \qquad \,z_0 =\xi - \mathbb{E}[\xi],\\
            d\lambda_{t} = -\left[\left(A_{t} - B_{t}R^{-1}_{t}S^{\intercal}_{t}\right)^{\intercal}\lambda_{t} + \left(Q_{t} - S_{t}R^{-1}_{t}S^{\intercal}_{t}\right)z_{t}\right]dt + \beta_{t}dW_{t} + \beta_t^{0}dW^{0}_{t}, \quad \lambda_{T} = Q_{T}z_{T},
        \end{cases}
    \end{equation}
    which are obtained by substituting \eqref{obtained-control} in \eqref{BrevState} and \eqref{obtained-adjoint}. Now, since the admissible control set $\breve{\mathcal{U}}: = \big\{ \alpha \in L_{\mathcal{F}}^{2}\left( [0,T],\mathbb{R}^{d}\right):\allowdisplaybreaks \E\left[\alpha_{t} \big| \mathcal{F}^{0}_{t}\right]=0,\,dt\otimes d\mathbb{P}-a.s. \big\}$ is constrained, we examine the admissibility of the obtained optimal control $\alpha^\ast$. First, subject to the conditions imposed on the relevant coefficients in \hyperref[MTProb]{\color{black}\thnameref{MTProb}}, the system \eqref{OptimBrevHamil:proof} is well defined and the optimal control $\alpha^*$ given by \eqref{obtained-control} belongs to $L_{\mathcal{F}}^{2}\left( [0,T],\mathbb{R}^{d}\right)$. Next, we show that $\E\left[\alpha^*_t \big | \mathcal{F}^{0}_{t}\right] = 0,\, dt\otimes d\mathbb{P}-a.s.$ To this end, we adopt the ansatz $\lambda_{t} = \Pi_{t}z_{t}$, and apply Itô's lemma to derive the BSDE satisfied by $\Pi_{t}z_{t}$. Then, we substitute this ansatz into the FBSDE given by \eqref{OptimBrevHamil:proof} and match its backward component with the BSDE obtained for $\Pi_tz_t$, which yields the stochastic Riccati equations satisfied by $\Pi_t$, given by    \begin{align}\label{Ricatti}
        d\Pi_{t} =& -\left[\Pi_{t}\left(A_{t} - B_{t}R^{-1}_{t}S^{\intercal}_{t}\right) + \left(A_{t} - B_{t}R^{-1}_{t}S^{\intercal}_{t}\right)^{\intercal}\Pi_{t} - \Pi_{t}B_{t}R_{t}^{-1}B_{t}^{\intercal}\Pi_{t} + Q_{t} - S_{t}R_{t}^{-1}S_{t}^{\intercal}\right]dt\notag\\
        &+\Psi_{t}^{0}dW_{t}^{0},\quad \Pi_{T} = Q_{T},
    \end{align} and 
    \begin{equation}
            \beta_{t} = \Pi_{t}D_{t},\quad \beta_t^{0} = \Psi^{0}_{t}z_{t}.
    \end{equation} 
    Moreover, after substituting $\alpha^{*}_{t} = -R^{-1}_{t}\left(S_{t}^{\intercal}+ B_{t}^{\intercal}\Pi_t\right)z_{t}$, the forward component in \eqref{OptimBrevHamil:proof} becomes
    \begin{equation}
        dz_{t} = \left(A_{t} - B_{t}R^{-1}_{t}{(S_{t}+\Pi_{t}B_{t})}^{\intercal}\right)z_{t}dt + D_{t} dW_{t}, \quad z_{0} = \xi - \E[\xi],
    \end{equation} 
    and hence, we have  
        \begin{equation}\label{OptimBrevHamilExpec:proof}
            d\E\left[z_{t} \big | \mathcal{F}^{0}_{t}\right] = \left(A_{t} - B_{t}R^{-1}_{t}{(S_{t}+\Pi_{t}B_{t})}^{\intercal}\right)\E\left[z_{t} \big | \mathcal{F}^{0}_{t}\right]dt,\quad \E\left[ z_{0}\right] = 0,
    \end{equation} 
    According to \cite{Tang2003}, under \Cref{A2}, the stochastic Riccati equation given by \eqref{Ricatti} admits a unique solution $\left(\Pi,\Psi^{0}\right)\in L_{\mathcal{F}^{0}}^{\infty}\left( [0,T],\mathbb{S}^{n}\right)\times L_{\mathcal{F}^{0}}^{2}\left([0,T],\mathbb{R}^{n}\right)$ with $\Pi_{t}\geq 0\,\,dt\otimes d\mathbb{P}-a.s.$ Thus, \eqref{OptimBrevHamilExpec:proof} is well defined and admits the unqiue solution given by
    \begin{equation}
        \E\left[ z_{t} \big | \mathcal{F}^{0}_{t}\right] = 0,\quad dt \otimes d\mathbb{P}-a.s.\label{sol1}
        \end{equation} 
    Consequently, we have    
    \begin{align}
     &\E\left[\alpha^*_t \big | \mathcal{F}^{0}_{t}\right]=-R^{-1}_{t}\left(S_{t}^{\intercal}+ B_{t}^{\intercal}\Pi_t\right)\E\left[z_{t} \big | \mathcal{F}^{0}_{t}\right]=0, \qquad 
        dt\otimes d\mathbb{P}-a.s,\\
    & \E\left[\beta^0_t \big | \mathcal{F}^{0}_{t}\right] = \Psi^0_t\E\left[z_{t} \big | \mathcal{F}^{0}_{t}\right] = 0, \qquad\qquad\qquad\qquad\quad\,\, dt\otimes d\mathbb{P}-a.s, \label{sol2}
    \end{align}
    which establishes that both the optimal control $\alpha^*$ and the  processes  $\left(\lambda,\beta,\beta^{0}\right)$ lie within their respective constrained admissible sets, i.e. $\breve{\mathcal{U}}: = \left\{ \alpha \in L_{\mathcal{F}}^{2}\left( [0,T],\mathbb{R}^{d}\right):\allowdisplaybreaks \E\left[\alpha_{t} \big| \mathcal{F}^{0}_{t}\right]=0,\,dt\otimes d\mathbb{P}-a.s. \right\}$ and $\mathcal{M}_{\mathcal{F}}\left([0,T],\mathbb{R}^{n}\right):=\left\{\left(\lambda,\beta,\beta^{0}\right)\in\mathcal{C}_{\mathcal{F}}\left([0,T],\mathbb{R}^{n}\right)\times L_{\mathcal{F}}^{2}\left( [0,T],\mathbb{R}^{n}\right) \times L_{\mathcal{F}}^{2}\left( [0,T],\mathbb{R}^{n}\right): \E\left[\beta^{0}_{t} \big| \mathcal{F}^{0}_{t}\right]=0,\,dt\otimes d\mathbb{P}-a.s.\right\}$, respectively.    
        
    We now follow a similar procedure to characterize the optimal control for  \hyperref[DecouplProb1]{\color{black}\thnameref{DecouplProb1}}. Consider the constant $\mu>0$ and any two admissible $v^{*},v\in\bar{\mathcal{U}}$. For every $t\in[0,T]$, the state process $y^{v^{*} + \mu v}$, satisfying \eqref{BarState} under the control input $v^{*} + \mu v$, is decomposed as 
    \begin{equation}
        y^{v^{*} + \mu v}_{t} = y_{t} + \mu y^{0}_{t}.
    \end{equation}
    Subsequently, the cost functional $\bar{J}\left(v^{*} + \mu v\right)$ is decomposed as  
    \begin{align}
            \bar{J}\left(v^{*} + \mu v\right) &= \bar{J}\left(v^{*}\right) + \frac{\mu^{2}}{2}\E\left\{\int_{0}^{T}\left[{y^{0}_{t}}^{\intercal}\bar{Q}_{t}y^{0}_{t} + 2{y^{0}_{t}}^{\intercal}\bar{S}_{t}v_{t} + {v_{t}}^{\intercal}R_{t}v_{t}\right]dt + {y^{0}_{T}}^{\intercal}\bar{Q}_{T}y^{0}_{T}\right\}\nonumber\allowdisplaybreaks\\
            &\quad + \mu \E\int_{0}^{T}v^{\intercal}_{t}\left[R_{t}v^{*}_{t} + \bar{S}_{t}^{\intercal}y_{t} + B_{t}^{\intercal}p_{t} + \varpi_{t}\right]dt,
    \end{align} where $y,y^{0}\in\mathcal{C}_{\mathcal{F}^{0}}\left([0,T],\mathbb{R}^{n}\right)$ and $(p,q)\in L_{\mathcal{F}^{0}}^{2}\left( [0,T],\mathbb{R}^{n}\right) \times L_{\mathcal{F}^{0}}^{2}\left( [0,T],\mathbb{R}^{n}\right)$, respectively, satisfy
        \begin{align}
            &dy_{t} = [(A_{t} + F_{t})y_{t} + B_{t}v^{*}_{t} + b_{t}]dt + D^{0}_{t} dW_{t}^{0},\qquad \qquad \quad \; y_{0} =\bar{\xi},\nonumber\allowdisplaybreaks\\
            &dy^{0}_{t} = [(A_{t} + F_{t})y^{0}_{t} + B_{t}v_{t}]dt, \qquad \qquad\qquad \qquad\qquad \quad \;\;\; y^{0}_{0} = 0,\label{P1-eqs}\allowdisplaybreaks\\
            &dp_{t} = -\left[(A_{t} + F_{t})^{\intercal}p_{t} + \bar{Q}_{t}y_{t} + \bar{S}_{t}v^{*}_{t} + \bar{\zeta}_{t}\right]dt + q_{t}dW_{t}^{0},\quad p_{T} = \bar{Q}_{T}y_{T}\nonumber.
        \end{align}
    
    Then, under \Cref{A2} and from \eqref{UnifCovexCondBar}, a control input $v^{*}$ is a minimizer for $\bar{J}(v)$ if and only if 
    \begin{equation}\label{optim:cond:bar}
         R_{t}v^{*}_{t} + \bar{S}_{t}^{\intercal}y_{t} + B_{t}^{\intercal}p_{t} + \varpi_{t} = 0,\quad  dt\otimes d\mathbb{P}-a.s.
    \end{equation} Subsequently, the optimal control is given by 
    \begin{equation}
        v^{*}_{t} = -R^{-1}_{t}\left(\bar{S}_{t}^{\intercal}y_{t} + B_{t}^{\intercal} p_{t} + \varpi_{t}\right),\quad dt\otimes d\mathbb{P}-a.s,
    \end{equation} 
    where $(y,p,q)\in \mathcal{C}_{\mathcal{F}^{0}}\left([0,T],\mathbb{R}^{n}\right)\times L_{\mathcal{F}^{0}}^{2}\left( [0,T],\mathbb{R}^{n}\right) \times L_{\mathcal{F}^{0}}^{2}\left( [0,T],\mathbb{R}^{n}\right)$ is the adapted solution of the FBSDEs given by
    \begin{equation}\label{OptimBarHamil:proof}
        \begin{cases}
            dy_{t} = \left[\left(A_{t} + F_{t} - B_{t}R^{-1}_{t}\bar{S}^{\intercal}_{t}\right)y_{t} - B_{t}R^{-1}_{t}B^{\intercal}_{t}p_{t} - B_{t}R^{-1}_{t}\varpi_{t} + b_{t}\right]dt + D_{t}^{0} dW_{t}^0, \qquad \qquad \quad y_{0} =\E[\xi],\\
            dp_{t} = -\left[\left(A_{t} + F_{t} - B_{t}R^{-1}_{t}\bar{S}^{\intercal}_{t}\right)^{\intercal}p_{t} + \left(\bar{Q}_{t} - \bar{S}_{t}R^{-1}_{t}\bar{S}^{\intercal}_{t}\right)y_{t} - \bar{S}_{t}R^{-1}_{t}\varpi_{t} + \bar{\zeta}_{t}\right]dt + q_{t}dW_{t}^{0},\quad p_{T} = \bar{Q}_{T}y_{T}.
        \end{cases}
    \end{equation} 
    Given the conditions imposed on the relevant coefficients in 
\hyperref[MTProb]{\color{black}\thnameref{MTProb}}, 
the system \eqref{OptimBarHamil:proof} is well defined, which in turn ensures the admissibility of \(v^{*}\); that is, 
\(v^{*} \in \bar{\mathcal{U}}:= L_{\mathcal{F}^{0}}^{2}\!\left([0,T], \mathbb{R}^{d}\right)\).
    
    To conclude, according to \Cref{KeyTheo}, and from \eqref{optim:cond:brev}-\eqref{OptimBrevHamil:proof} and \eqref{P1-eqs}-\eqref{optim:cond:bar}, an admissible pair $(x^{*},\,u^{*})\in\mathcal{C}_{\mathcal{F}}\left([0,T],\mathbb{R}^{n}\right)\times\mathcal{U}$ is an optimal pair for \hyperref[MTProb]{\color{black}\thnameref{MTProb}}, if and only if the pair  $(\bar{x}^{*}_{t},\,\bar{u}^{*}_{t}) = \left(\E[x^{*}_{t}|\mathcal{F}_{t}^0],\,\E[u^{*}_{t}|\mathcal{F}_{t}^0]\right)\in \mathcal{C}_{\mathcal{F}^{0}}\left([0,T],\mathbb{R}^{n}\right)\times\bar{\mathcal{U}}$ satisfies \eqref{opt-control-prob1}-\eqref{BarFBSDE2} and the pair  $(\breve{x}^{*}_{t},\, \breve{u}^{*}_{t}) = \left(x^{*}_{t} - \E[x^{*}_{t}|\mathcal{F}_{t}^0],\,u^{*}_{t} - \E[u^{*}_{t}|\mathcal{F}_{t}^0]\right)\in \mathcal{C}_{\mathcal{F}}\left([0,T],\mathbb{R}^{n}\right)\times\breve{\mathcal{U}}$ satisfies \eqref{opt-control-prob2}-\eqref{BrevFBSDE2}. This completes the proof.
\end{proof} 

The following Corollary characterizes the optimal control for \hyperref[MTProb]{\color{black}\thnameref{MTProb}} through two decoupled sets of classical linear FBSDEs. 
\begin{corollary}(Optimal Control)\label{optimal-control-MFT}
 Under 
    \Cref{A2}, the optimal control for \hyperref[MTProb]{\color{black}\thnameref{MTProb}} is given by 
    \begin{equation}\label{OptimControl}
        u^{*}_{t} = -R^{-1}_{t}\left[\bar{S}_{t}^{\intercal}\bar{x}^{*}_{t} + S_{t}^{\intercal}\breve{x}^{*}_{t} + B_{t}^{\intercal}\left(\lambda^{*}_{t} + p^{*}_{t}\right) + \varpi_{t}\right],\quad dt\otimes d\mathbb{P}-a.s,
    \end{equation} 
    where $(\bar{x}^{*},p^{*},q^{*})\in \mathcal{C}_{\mathcal{F}^{0}}\left([0,T],\mathbb{R}^{n}\right)\times L_{\mathcal{F}^{0}}^{2}\left( [0,T],\mathbb{R}^{n}\right) \times L_{\mathcal{F}^{0}}^{2}\left( [0,T],\mathbb{R}^{n}\right)$ is the adapted solution of the coupled FBSDEs given by
    \begin{equation}\label{OptimBarHamil}
        \begin{cases}
            d\bar{x}^{*}_{t} = \left[\left(A_{t} + F_{t} - B_{t}R^{-1}_{t}\bar{S}^{\intercal}_{t}\right)\bar{x}^{*}_{t} - B_{t}R^{-1}_{t}B^{\intercal}_{t}p^{*}_{t} - B_{t}R^{-1}_{t}\varpi_{t} + b_{t}\right]dt + D_{t}^{0} dW_{t}^0, \qquad \qquad \quad \bar{x}^{*}_0 =\E[\xi],\allowdisplaybreaks\\
            dp^{*}_{t} = -\left[\left(A_{t} + F_{t} - B_{t}R^{-1}_{t}\bar{S}^{\intercal}_{t}\right)^{\intercal}p^{*}_{t} + \left(\bar{Q}_{t} - \bar{S}_{t}R^{-1}_{t}\bar{S}^{\intercal}_{t}\right)\bar{x}^{*}_{t} - \bar{S}_{t}R^{-1}_{t}\varpi_{t} + \bar{\zeta}_{t}\right]dt + q^{*}_{t}dW_{t}^{0},\quad p^{*}_{T} = \bar{Q}_{T}\bar{x}^{*}_{T},
        \end{cases}
    \end{equation} and $(\breve{x}^{*},\lambda^{*},\beta^{*},\beta_t^{0,*})\in \mathcal{C}_{\mathcal{F}}\left([0,T],\mathbb{R}^{n}\right)\times\mathcal{M}_{\mathcal{F}}\left([0,T],\mathbb{R}^{n}\right)$ is the adapted solution of the coupled FBSDEs given by
    \begin{equation}\label{OptimBrevHamil}
        \begin{cases}
            d\breve{x}^{*}_{t} = \left[\left(A_{t} - B_{t}R^{-1}_{t}S^{\intercal}_{t}\right)\breve{x}^{*}_{t} - B_{t}R^{-1}_{t}B^{\intercal}_{t}\lambda^{*}_{t}\right]dt + D_{t} dW_{t},\qquad \qquad \qquad \qquad \qquad \,\breve{x}^{*}_0 = \xi - \E[\xi],\\
            d\lambda^{*}_{t} = -\left[\left(A_{t} - B_{t}R^{-1}_{t}S^{\intercal}_{t}\right)^{\intercal}\lambda^{*}_{t} + \left(Q_{t} - S_{t}R^{-1}_{t}S^{\intercal}_{t}\right)\breve{x}^{*}_{t}\right]dt + \beta^{*}_{t}dW_{t} + \beta_t^{0,*}dW^{0}_{t}, \quad \lambda^{*}_{T} = Q_{T}\breve{x}^{*}_{T}.
        \end{cases}
    \end{equation}
\end{corollary}
\begin{proof}
     According to \Cref{UnifConvChara}, the cost functional $J(u)$ of \hyperref[MTProb]{\color{black}\thnameref{MTProb}} is uniformly convex. Consequently, there exists a unique $(x^{*},\,u^{*})\in\mathcal{C}_{\mathcal{F}}\left([0,T],\mathbb{R}^{n}\right)\times\mathcal{U}$ such that 
    \begin{equation}
        J\left(u^{*}\right) \leq J\left(u\right), \quad \forall u \in \mathcal{U}.
    \end{equation} 
   Since $u^*$ can be decomposed as $u^*=\bar{u}^{*}+\breve{u}^{*}$, and according to \Cref{OptimPrincMFT}, $\bar{u}^{*}\in \bar{\mathcal{U}}$ satisfies \eqref{opt-control-prob1} and $\breve{u}^{*}\in \breve{\mathcal{U}}$ satisfies \eqref{opt-control-prob2}, the optimal control $u^*_t$ is given by     
    \begin{equation}
        u^{*}_{t} = \bar{u}^{*}_{t} + \breve{u}^{*}_{t} = -R^{-1}_{t}\left[\bar{S}_{t}^{\intercal}\bar{x}^{*}_{t} + S_{t}^{\intercal}\breve{x}^{*}_{t} + B_{t}^{\intercal}\left(\lambda^{*}_{t} + p^{*}_{t}\right) + \varpi_{t}\right],\quad dt\otimes d\mathbb{P}-a.s.
    \end{equation}  Substituting $\bar{u}^{*}$ (respectively $\breve{u}^{*}$) in the FBSDEs given by  \eqref{BarFBSDE1}-\eqref{BarFBSDE2} (respectively \eqref{BrevFBSDE1}-\eqref{BrevFBSDE2}) yields \eqref{OptimBarHamil} (respectively \eqref{OptimBrevHamil}).
\end{proof}

    \begin{remark}(Consistency with Existing Results) We show that the results presented in \Cref{optimal-control-MFT}, obtained using the proposed decomposition method, are consistent with those reported in existing literature; specifically, \cite[Theorem 3.1]{Yong2013}, \cite[Proposition 2.4]{Graber2016} and \cite[Theorem 3.1]{Yang2022}, which address LQ McKean-Vlasov control problems with deterministic coefficients, driven by one or two Wiener processes, using an extended form of the stochastic maximum principle or variational methods. For comparison purposes, we first express the cost functional $J(u)$ in \eqref{MTCost} in a form analogous to that used in \citep{Yong2013,Graber2016,Yang2022}, given by 
        \begin{align}
            J\left(u\right) = \frac{1}{2}\E\bigg\{\int_{0}^{T}\Big[&x^{\intercal}_{t}Q_{t}x_{t} + \E[x_{t}|\mathcal{F}_{t}^0]^{\intercal}\left(\bar{Q}_{t} - Q_{t}\right)\E[x_{t}|\mathcal{F}_{t}^0] + u^{\intercal}_{t}R_{t}u_{t} + 2x^{\intercal}_{t}S_{t}u_{t} - 2\E[x_{t}|\mathcal{F}_{t}^0]^{\intercal}H^{\intercal}S_{t}\E[u_{t}|\mathcal{F}_{t}^0]\notag\\ 
            & + 2\bar{\zeta}^{\intercal}_{t}\E[x_{t}|\mathcal{F}_{t}^0] + 2\varpi_{t}^{\intercal}u_{t} \Big]dt + x^{\intercal}_{T}Q_{T}x_{T} + \E[x_{T}|\mathcal{F}_{T}^0]^{\intercal}\left(\bar{Q}_{T} - Q_{T}\right)\E[x_{T}|\mathcal{F}_{T}^0]\bigg\}.
        \end{align} 
  On the one hand, following the extended stochastic maximum principle approach in \citep{Yong2013,Graber2016,Yang2022}, the optimal pair $(x^\circ,u^{\circ})$ for \hyperref[MTProb]{\color{black}\thnameref{MTProb}} satisfies 
    \begin{equation}\label{stoch-max-1}
        \begin{split}
            &u^{\circ}_{t} = -R^{-1}_{t}\left[S_{t}^{\intercal}x^\circ_{t} - S_{t}^{\intercal}H\E[x^\circ_{t}|\mathcal{F}_{t}^0] + B_{t}^{\intercal}y_{t}^\circ + \varpi_{t}\right],\quad dt\otimes d\mathbb{P}-a.s,\\
            &dx^\circ_{t} = [A_{t}x^\circ_{t}+B_{t}u^{\circ}_{t}+ F_{t}\E[x^\circ_{t}|\mathcal{F}_{t}^0] + b_{t}]dt+D_{t} dW_{t} + D_{t}^{0} dW_{t}^{0} ,  \quad x^\circ_0 = \xi,
        \end{split}
    \end{equation} 
    where the triple $(y^\circ, z^{\circ}, z^{0,\circ})$ satisfies
    \begin{equation}\label{stoch-max-2}
        \begin{split}
            dy_{t}^{\circ} = &-\left[A^{\intercal}_{t}y_{t}^{\circ} + F^{\intercal}_{t}\E[y_{t}^{\circ}|\mathcal{F}_{t}^0] +  Q_{t}x^{\circ}_{t} +  (\bar{Q}_{t} - Q_{t})\E[x^{\circ}_{t}|\mathcal{F}_{t}^0] + S_{t}u^{\circ}_{t} - H^{\intercal}S_{t}\E[u^{\circ}_{t}|\mathcal{F}_{t}^0] + \bar{\zeta}_{t}\right]dt\\ 
             &+ z_{t}^{\circ} dW_{t} +z_{t}^{0,\circ}dW^{0}_{t}, \quad y^{\circ}_{T} = Q_{T}x^\circ_{T} + (\bar{Q}_{T} - Q_{T})\E[x^{\circ}_{T}|\mathcal{F}_{T}^0].
        \end{split}
    \end{equation}
    On the other hand, using the proposed decomposition method, from \Cref{OptimPrincMFT} and \Cref{optimal-control-MFT}, the optimal pair $(x^{*},u^{*})=(\bar{x}^{*}+\breve{x}^{*},\bar{u}^{*}+\breve{u}^{*})$ for \hyperref[MTProb]{\color{black}\thnameref{MTProb}} satisfies
        \begin{align}\label{decompose-1}
            &u^{*}_{t} = -R^{-1}_{t}\left[S_{t}^{\intercal}x^{*}_{t} - S_{t}^{\intercal}H\bar{x}^{*}_{t} + B_{t}^{\intercal}\left(\lambda^{*}_{t} + p^{*}_{t}\right) + \varpi_{t}\right],\quad dt\otimes d\mathbb{P}-a.s,\allowdisplaybreaks\\
            &dx^{*}_{t} = [A_{t}x^{*}_{t} + B_{t}u^{*}_{t}+ F_{t}\bar{x}^{*}_{t} + b_{t}]dt + D_{t} dW_{t} + D_{t}^{0} dW_{t}^{0} ,  \quad x^{*}_{0} = \xi,\nonumber
        \end{align}
    where the triple $(\lambda^*+p^*, \beta^*, \beta^{0,*} + q^{*})$ satisfies         \begin{align}\label{decompose-2}
            d\left(\lambda^{*}_{t} + p^{*}_{t}\right) =& -\left[A^{\intercal}_{t}\left(\lambda^{*}_{t} + p^{*}_{t}\right) + F^{\intercal}_{t}p^{*}_{t} +  Q_{t}x^{*}_{t} +  (\bar{Q}_{t} - Q_{t})\bar{x}^{*}_{t} + S_{t}u^{*}_{t} - H^{\intercal}S_{t}\bar{u}^*_t + \bar{\zeta}_{t}\right]dt\nonumber\allowdisplaybreaks\\ 
            & + \beta^{*}_{t}dW_{t} + \left(\beta_t^{0,*} + q^{*}_{t}\right)dW^{0}_{t}, \quad \lambda^{*}_{T} + p^{*}_{T} = Q_{T}{x}^{*}_{T} + (\bar{Q}_{T} - Q_{T})\bar{x}^{*}_{T},
        \end{align}
  and where $\bar{x}_t^* =\mathbb{E}[x^*_t|\mathcal{F}^0_t]$, since $\mathbb{E}[x^*_t|\mathcal{F}^0_t]=\mathbb{E}[\breve{x}^*_t+\bar{x}^*_t|\mathcal{F}^0_t]=\bar{x}^*_t$ by the definitions of {\color{black}\thnameref{DecouplProb1}} and {\color{black}\thnameref{DecouplProb2}}. Similarly, $\bar{u}^*_t=\E[u^{*}_{t}|\mathcal{F}_{t}^0]$ and $p^*_t=\E[\lambda^*_t+p^*_t|\mathcal{F}_{t}^0]$. We observe that \eqref{stoch-max-1}–\eqref{stoch-max-2} coincide with \eqref{decompose-1}–\eqref{decompose-2}. By \Cref{thm:wellposed-gen-FBSDE} and \Cref{wellposed_our_FBSDE} (see \ref{AppendixC}), these FBSDEs admit a unique solution. Consequently, $(x^\circ_t,y^\circ_t,z^\circ_t,z^{0,\circ}_t)=(x^*_t,\lambda^*_t+p^*_t,\beta^*_t, \beta^{0,*}_t + q^{*}_t),\, dt\otimes d\mathbb{P}-a.s$. Hence, $({x}^*_t,{u}^*_t)=({x}^\circ_t,{u}^\circ_t), \, dt\otimes d\mathbb{P}-a.s$.  
\end{remark}   
 For completeness, the following theorem presents the obtained optimal control for \hyperref[MTProb]{\color{black}\thnameref{MTProb}} in feedback form.
    
    \begin{theorem}(Feedback Representation)\label{thm:feedback-rep} Suppose that \Cref{A2} holds. Then, the optimal control for \hyperref[MTProb]{\color{black}\thnameref{MTProb}}, given by \eqref{OptimControl}-\eqref{OptimBrevHamil}, admits a feedback form as    \begin{equation}\label{MTFedBackOptim}
                u^{*}_{t} = -R^{-1}_{t}\left[(\Pi_{t}B_{t}+S_{t})^{\intercal}(x_t^*-\bar{x}_t^*)  + (L_{t}B_{t} + \bar{S}_{t})^{\intercal}\bar{x}^{*}_{t}  + B^{\intercal}_{t}\ell_{t} + \varpi_{t}\right],
            \end{equation}   
    where ${x}^{*}_t$ and $\bar{x}^{*}_{t}$, respectively, satisfy
     \begin{align}
            dx^{*}_{t} =& \big[\left(A_{t} - B_{t}R^{-1}_{t}(\Pi_{t}B_{t}+S_{t})^{\intercal}\right)x_t^* + \left(F_{t} - B_{t}R^{-1}_{t}((L_{t} - \Pi_{t})B_{t} - H^{\intercal}S_{t})^{\intercal}\right)\bar{x}^{*}_{t}\nonumber\\
            &\,\,\,  - B_{t}R^{-1}_{t}\left(B^{\intercal}_{t}\ell_{t} + \varpi_{t}\right)  + b_{t}\big]dt+ D_{t} dW_{t} + D_{t}^{0} dW_{t}^0, \,\,\qquad \qquad \qquad \qquad \qquad \qquad \qquad
            \; \, x^{*}_{0} = \xi,\\
           d\bar{x}^{*}_{t} =& \left[\left(A_{t} + F_{t} - B_{t}R^{-1}_{t}(L_{t}B_{t} + \bar{S}_{t})^{\intercal}\right)\bar{x}^{*}_{t} - B_{t}R^{-1}_{t}\left(B^{\intercal}_{t}\ell_{t} + \varpi_{t}\right)  + b_{t}\right]dt+ D_{t}^{0} dW_{t}^0,  \qquad \quad \,
            \bar{x}^{*}_{0} =\E[\xi].\allowdisplaybreaks
    \end{align}
            Furthermore,  $\left(L,\bar{\Psi}^{0}\right)\in L_{\mathcal{F}^{0}}^{\infty}\left( [0,T],\mathbb{S}^{n}\right)\times L_{\mathcal{F}^{0}}^{2}\left([0,T],\mathbb{R}^{n}\right)$ and $\left(\Pi,\Psi^{0}\right)\in L_{\mathcal{F}^{0}}^{\infty}\left( [0,T],\mathbb{S}^{n}\right)\times L_{\mathcal{F}^{0}}^{2}\left([0,T],\mathbb{R}^{n}\right)$ are, respectively, the unique solutions of the stochastic Riccati equations 
        \begin{align}      
            d\Pi_{t} =& -\left[\Pi_{t}\left(A_{t} - B_{t}R^{-1}_{t}S^{\intercal}_{t}\right) + \left(A_{t} - B_{t}R^{-1}_{t}S^{\intercal}_{t}\right)^{\intercal}\Pi_{t} - \Pi_{t}B_{t}R_{t}^{-1}B_{t}^{\intercal}\Pi_{t} + Q_{t} - S_{t}R_{t}^{-1}S_{t}^{\intercal}\right]dt\notag\\
             &+\Psi_{t}^{0}dW_{t}^{0},\quad \Pi_{T} = Q_{T},\label{SRE}\allowdisplaybreaks\\   
            dL_{t} =& -\left[L_{t}\left(A_{t} + F_{t} - B_{t}R_{t}^{-1}\bar{S}_{t}^{\intercal}\right) + \left(A_{t} + F_{t} - B_{t}R_{t}^{-1}\bar{S}_{t}^{\intercal}\right)^{\intercal}L_{t} - L_{t}B_{t}R_{t}^{-1}B_{t}^{\intercal}L_{t}+\bar{Q}_{t} - \bar{S}_{t}R_{t}^{-1}\bar{S}_{t}^{\intercal}\right]dt\notag\\
            & + \bar{\Psi}_{t}^{0}dW_{t}^{0}, \quad  L_{T} = \bar{Q}_{T},  \label{SRE1}
    \end{align}
            and $\left(\ell,\psi^{0}\right)\in \mathcal{C}_{\mathcal{F}^{0}}\left([0,T],\mathbb{R}^{n}\right)\times L_{\mathcal{F}^{0}}^{2}\left( [0,T],\mathbb{R}^{n}\right)$ is the unique solution of the BSDE 
        \begin{align}\label{offset}
                d\ell_{t} &= -\left[\left(A_{t} + F_{t} - B_{t}R_{t}^{-1}(L_{t}B_{t} + \bar{S}_{t})^{\intercal}\right)^{\intercal}\ell_{t} - \left(L_{t}B_{t} + \bar{S}_{t}\right)R_{t}^{-1}\varpi_{t} + \bar{\zeta}_{t} + L_{t}b_{t} + \bar{\Psi}_{t}^{0}D_{t}^{0}\right]dt\notag\\
             &\quad + \psi_{t}^{0}dW_{t}^{0},\quad \ell_{T}=0.
        \end{align}
    \end{theorem}
    \begin{proof}
        Adopting the ansatzes $p^{*}_{t} = L_{t}\bar{x}^{*}_{t} + \ell_{t}$ and $\lambda^{*}_{t} = \Pi_{t}\breve{x}^{*}_{t}$, and applying Itô's lemma to these expressions, we obtain the BSDEs that they satisfy.  We then substitute these expression in \eqref{OptimBarHamil} and \eqref{OptimBrevHamil} and match them with the corresponding SDEs obtained above, which yields  \eqref{SRE}-\eqref{offset} and 
        \begin{align*}\label{FeedForm}
            &q^{*}_{t} = \bar{\Psi}_{t}^{0}\bar{x}^{*}_{t} + L_{t}D_{t}^{0} + \psi_{t}^{0},\\
                &\beta^{*}_{t} = \Pi_{t}D_{t},\\  &\beta_t^{0,*} = \Psi^{0}_{t}\breve{x}^{*}_{t},
        \end{align*}
        where we can easily verify that $\mathbb{E}[\beta_t^{0,*}|\mathcal{F}^0_t]=0$ for every $t\in[0,T]$, and hence $(\lambda^*, \beta^*,\beta^{0,*} )\in \mathcal{M}_{\mathcal{F}}\left([0,T],\mathbb{R}^{n}\right)$. According to \cite[Theorem 5.3, pp.~64]{Tang2003}, under \Cref{A2} and the imposed conditions on the system coefficients, the stochastic Riccati equations \eqref{SRE} and \eqref{SRE1} each admit a unique solution, $\left(\Pi,\Psi^{0}\right),(L,\bar{\Psi}^{0})\in L_{\mathcal{F}^{0}}^{\infty}\left( [0,T],\mathbb{S}^{n}\right)\times L_{\mathcal{F}^{0}}^{2}\left([0,T],\mathbb{R}^{n}\right)$ with $\Pi_{t},L_{t}\geq 0\,\,dt\otimes d\mathbb{P}-a.s.$ Subsequently, the BSDE \eqref{offset} also admits a unique solution $\left(\ell,\psi^{0}\right)\in \mathcal{C}_{\mathcal{F}^{0}}\left([0,T],\mathbb{R}^{n}\right)\times L_{\mathcal{F}^{0}}^{2}\left( [0,T],\mathbb{R}^{n}\right)$. Finally, substituting these ansatzes in   \eqref{OptimControl} yields the feedback representation given by \eqref{MTFedBackOptim}.
    \end{proof}

\section{Conclusion}\label{sec:conclusion}
In this paper, we have introduced a decomposition approach for studying LQ McKean-Vlasov optimal control problems involving conditional expectations of the state and random coefficients. Through this approach, the considered McKean-Vlasov control problem is shown to be equivalent to two stochastic optimal control problems that can be addressed using standard methods. Furthermore, we applied a variational method to the auxiliary problems and represented the optimal solution of the original McKean-Vlasov problem as the sum of the optimal solutions to these two problems. This, in turn, yields an optimal control in a state feedback form characterized by two decoupled standard stochastic Riccati differential equations. Finally, we showed that standard dynamic programming can also be applied to address the auxiliary problems.

\renewcommand{\bibname}{References}  
\bibliographystyle{elsarticle-num-names} 
\bibliography{Biblio}

\appendix 

\section{Proof of \Cref{lemma:cost-term-decomp}}\label{DecomProof}
We prove properties (iv) and (v) in the following sections. 
\subsection{Proof of Property (iv)}
We aim to show that $\E[\left(x_{t} - H\bar{x}_{t}\right)^{\intercal}S_{t}u_{t}] = \E[\breve{x}_{t}^{\intercal}S_{t}\breve{u}_{t}] + \E[\bar{x}^{\intercal}_{t}\bar{S}_{t}\bar{u}_{t}]$. To this end, we first express $\E[\breve{x}_{t}^{\intercal}S_{t}\breve{u}_{t}]$ as the sum of three terms as in   \begin{equation}\label{A:Eq1}
            \begin{split}
                \E[\breve{x}_{t}^{\intercal}S_{t}\breve{u}_{t}] &= \E\left[\left(x_{t} - \bar{x}_{t}\right)^{\intercal}S_{t}\left(u_{t} - \bar{u}_{t}\right)\right]\\
                &= \E\left[\left(x_{t} - H\bar{x}_{t} +  H\bar{x}_{t} - \bar{x}_{t}\right)^{\intercal}S_{t}\left(u_{t} - \bar{u}_{t}\right)\right]\\
                &= \E\left[\left(x_{t} - H\bar{x}_{t}\right)^{\intercal}S_{t}u_{t}\right] - \underbrace{\E\left[\left(x_{t} - H\bar{x}_{t}\right)^{\intercal}S_{t} \bar{u}_{t}\right]}_{D_1} - \underbrace{\E\left[\left(\bar{x}_{t} - H\bar{x}_{t}\right)^{\intercal}S_{t}\left(u_{t} - \bar{u}_{t}\right)\right]}_{D2},
            \end{split}
        \end{equation}
        where 
        \begin{equation}
            \begin{split}
                D_1=\E\left[\left(x_{t} - H\bar{x}_{t}\right)^{\intercal}S_{t} \bar{u}_{t}\right] &= \E\left(\E\left[\left(x_{t} - H\bar{x}_{t}\right)^{\intercal}S_{t} \bar{u}_{t}\Big |\mathcal{F}^{0}_{t}\right]\right)\\
                &= \E\left(\E\left[\left(x_{t} - H\bar{x}_{t}\right)^{\intercal}\Big |\mathcal{F}^{0}_{t}\right] S_{t}\bar{u}_{t}\right)\\
                &= \E\left[\left(\bar{x}_{t} - H\bar{x}_{t}\right)^{\intercal} S_{t}\bar{u}_{t}\right]\\
                &= \E\left[\bar{x}^{\intercal}_{t}\left(I_{n} - H\right)^{\intercal} S_{t}\bar{u}_{t}\right]\\
                &= \E\left[\bar{x}^{\intercal}_{t}\bar{S}_{t}\bar{u}_{t}\right],\\
            \end{split}
        \end{equation}
        and         \begin{equation}\label{A:Eq2}
            \begin{split}
            D_2=\E\left[\left(\bar{x}_{t} - H\bar{x}_{t}\right)^{\intercal}S_{t}\left(u_{t} - \bar{u}_{t}\right)\right] &= \E\left(\E\left[\left(\bar{x}_{t} - H\bar{x}_{t}\right)^{\intercal}S_{t}\left(u_{t} - \bar{u}_{t}\right)\Big | \mathcal{F}^{0}_{t}\right]\right)\\
            &= \E\left(\left(\bar{x}_{t} - H\bar{x}_{t}\right)^{\intercal}S_{t}\E\left[u_{t} - \bar{u}_{t}\Big | \mathcal{F}^{0}_{t}\right]\right)\\
            &= 0.
            \end{split}
        \end{equation} 
        Hence, \eqref{A:Eq1} is expressed as
        \begin{equation}
            \E[\breve{x}_{t}^{\intercal}S_{t}\breve{u}_{t}] = \E\left[\left(x_{t} - H\bar{x}_{t}\right)^{\intercal}S_{t}u_{t}\right] - \E\left[\bar{x}^{\intercal}_{t}\bar{S}_{t}\bar{u}_{t}\right],
        \end{equation} which implies the property (iv).

\subsection{Proof of Property (v)}
Our goal is to establish that $\E[\left(x_{t} - H\bar{x}_{t}\right)^{\intercal}Q_{t}\left(x_{t} - H\bar{x}_{t}\right)] = \E\left[\breve{x}^{\intercal}_{t}Q_{t}\breve{x}_{t}\right] + \E\left[\bar{x}_{t}^{\intercal}\bar{Q}_{t}\bar{x}_{t} \right]$. Similar to the proof of property (iv), we first decompose $\E[\breve{x}^{\intercal}_{t}Q_{t}\breve{x}_{t}]$ into three terms, given by
    \begin{align}\label{A:Eq3}
            \E[\breve{x}^{\intercal}_{t}Q_{t}\breve{x}_{t}] &= \E\left[\left(x_{t} - \bar{x}_{t}\right)^{\intercal}Q_{t}\left(x_{t} - \bar{x}_{t}\right)\right]\nonumber\allowdisplaybreaks \\
            &= \E\left[\left(x_{t} - H\bar{x}_{t} + H\bar{x}_{t} - \bar{x}_{t}\right)^{\intercal}Q_{t}\left(x_{t} - H\bar{x}_{t} + H\bar{x}_{t} - \bar{x}_{t}\right)\right]\nonumber\allowdisplaybreaks \\
            &= \E\left[\left(x_{t} - H\bar{x}_{t}\right)^{\intercal}Q_{t}\left(x_{t} - H\bar{x}_{t}\right)\right] + \underbrace{\E\left[\left(\bar{x}_{t} - H\bar{x}_{t}\right)^{\intercal}Q_{t}\left(\bar{x}_{t} - H\bar{x}_{t}\right)\right]}_{D^\prime_1} - 2 \underbrace{\E\left[\left(x_{t} - H\bar{x}_{t}\right)^{\intercal}Q_{t}\left(\bar{x}_{t} - H\bar{x}_{t}\right)\right]}_{D^\prime_2},
        \end{align} 
        where
            \begin{align}
               D^\prime_1= \E\left[\left(\bar{x}_{t} - H\bar{x}_{t}\right)^{\intercal}Q_{t}\left(\bar{x}_{t} - H\bar{x}_{t}\right)\right] &= \E\left[\bar{x}^{\intercal}_{t}\left(I_{d} - H\right)^{\intercal}Q_{t}\left(I_{d} - H\right)\bar{x}_{t}\right]\nonumber\allowdisplaybreaks\\
            & = \E\left[\bar{x}^{\intercal}_{t}\bar{Q}_{t}\bar{x}_{t}\right],
            \end{align}
        and
        \begin{align}\label{A:Eq4}
              D^\prime_2=  \E\left[\left(x_{t} - H\bar{x}_{t}\right)^{\intercal}Q_{t}\left(\bar{x}_{t} - H\bar{x}_{t}\right)\right] &= \E\left(\E\left[\left(x_{t} - H\bar{x}_{t}\right)^{\intercal}Q_{t}\left(\bar{x}_{t} - H\bar{x}_{t}\right) \big| \mathcal{F}^{0}_{t}\right]\right)\nonumber\allowdisplaybreaks \\
                &= \E\left(\E\left[\left(x_{t} - H\bar{x}_{t}\right)^{\intercal} \big| \mathcal{F}^{0}_{t}\right]Q_{t}\left(\bar{x}_{t} - H\bar{x}_{t}\right)\right)\nonumber\allowdisplaybreaks \\
                &= \E\left[\left(\bar{x}_{t} - H\bar{x}_{t}\right)^{\intercal}Q_{t}\left(\bar{x}_{t} - H\bar{x}_{t}\right)\right].
            \end{align} 
        Subsequently, \eqref{A:Eq3} is expressed as 
        \begin{equation}
            \E[\breve{x}^{\intercal}_{t}Q_{t}\breve{x}_{t}] = \E\left[\left(x_{t} - H\bar{x}_{t}\right)^{\intercal}Q_{t}\left(x_{t} - H\bar{x}_{t}\right)\right] - \E\left[\bar{x}^{\intercal}_{t}\bar{Q}_{t}\bar{x}_{t}\right].
        \end{equation}
        Hence, property (v) holds.
        
\section{Solution of \color{black}\thnameref{DecouplProb2} via Dynamic Programming }\label{AppendixB}
In this section, we first demonstrate that the principle of optimality holds for \color{black}\thnameref{DecouplProb2}, and then use a standard stochastic Hamilton-Jacobi-Bellman (HJB) equation to derive the optimal control.

Specifically, for any fixed initial time $t\in[0,T]$ and initial state $\breve{\xi}\in\left\{\breve{\xi} \in L^{2}_{\mathcal{F}_t}\left(\Omega,\mathbb{R}^{n}\right): \E\big[\breve{\xi}|\mathcal{F}^{0}_{t}\big]=0\right\}$, we aim to find the control process $\alpha^\ast \in \breve{\mathcal{U}}_{t}:= \Big\{ \alpha \in L_{\mathcal{F}}^{2}\left( [t,T],\mathbb{R}^{d}\right): \E\left[\alpha_{\tau} \big| \mathcal{F}^{0}_{\tau}\right]=0,\,d\tau\otimes d\mathbb{P}-a.s. \Big\}$ that minimizes the cost functional given by
        \begin{equation}\label{Append:BrevCost}
              \breve{J}_{t,\breve{\xi}}\left(\alpha\right):=\frac{1}{2}\E\left[\int_{t}^{T}\left[{z_{\tau}}^{\intercal}Q_{\tau}z_{\tau}+2{z_{\tau}}^{\intercal}S_{\tau}\alpha_{\tau} + {\alpha_{\tau}}^{\intercal}R_{\tau}\alpha_{\tau}\right]d\tau + {z_{T}}^{\intercal}Q_{T}z_{T}\bigg| \mathcal{F}_{t}\right]
          \end{equation} where the state process $z \in \mathcal{C}_{\mathcal{F}}\left([t,T],\mathbb{R}^{n}\right)$ satisfies
            \begin{equation}\label{Append:BrevState}
                    dz_{\tau} = \left[A_{\tau}z_{\tau} + B_{\tau}\alpha_{\tau}\right]d\tau + D_{\tau} dW_{\tau}, \quad
                    z_{t} = \breve{\xi}.
               \end{equation}

          The associated value function for this problem is defined by
          \begin{equation}\label{Append:ValueFunc}
            V(t,\breve{\xi}):=\inf_{\alpha\in \breve{\mathcal{U}}_{t}}\breve{J}_{t,\breve{\xi}}\left(\alpha\right),\quad \mathbb{P}-a.s.,\quad \forall (t, \breve{\xi})\in[0,T)\times\left\{\breve{\xi} \in L^{2}_{\mathcal{F}_t}\left(\Omega,\mathbb{R}^{n}\right): \E\big[\breve{\xi}|\mathcal{F}^{0}_{t}\big]=0\right\}.
                  \end{equation}
          
\subsection{Principle of Optimality}\label{Append:DP} 
We begin with the following result. 
        \begin{theorem}
            For every $t\in[0,T),\,\breve{\xi}\in\left\{\breve{\xi} \in L^{2}_{\mathcal{F}_t}\left(\Omega,\mathbb{R}^{n}\right): \E\big[\breve{\xi}|\mathcal{F}^{0}_{t}\big]=0\right\}$ and   $h\in(t,T)$, the value function $V(.,.)$ associated with the problem defined by \eqref{Append:BrevCost}-\eqref{Append:BrevState} satisfies 
          \begin{equation}\label{Append:BelmanPrinc}
              V(t,\breve{\xi}) = \inf_{\alpha\in \breve{\mathcal{U}}_{t}}\E\left\{\frac{1}{2}\int_{t}^{h}\left[{z_{\tau}}^{\intercal}Q_{\tau}z_{\tau}+2{z_{\tau}}^{\intercal}S_{\tau}\alpha_{\tau} + {\alpha_{\tau}}^{\intercal}R_{\tau}\alpha_{\tau}\right]d\tau + V(h,z_{h})\bigg| \mathcal{F}_{t}\right\},\quad dt\otimes d\mathbb{P}-a.s.
          \end{equation} 
        \end{theorem} 

\begin{proof}
We split the proof into two steps.

          \textbf{Step 1.} 
          From \eqref{Append:BrevCost} and by the law of iterated conditional expectation, for every $\alpha\in \breve{\mathcal{U}}_{t}$, $ t\in[0,T)$, $\breve{\xi}\in\left\{\breve{\xi} \in L^{2}_{\mathcal{F}_t}\left(\Omega,\mathbb{R}^{n}\right): \E\big[\breve{\xi}|\mathcal{F}^{0}_{t}\big]=0\right\}$ and $h\in(t,T)$, we have
          \begin{align}
              \breve{J}_{t,\breve{\xi}}\left(\alpha\right) &= \frac{1}{2}\E\bigg\{\int_{t}^{h}\left[{z_{\tau}}^{\intercal}Q_{\tau}z_{\tau}+2{z_{\tau}}^{\intercal}S_{\tau}\alpha_{\tau} + {\alpha_{\tau}}^{\intercal}R_{\tau}\alpha_{\tau}\right]d\tau\nonumber\allowdisplaybreaks\\
              &\qquad \quad  +  \E\bigg[\int_{h}^{T}\big[{z_{\tau}}^{\intercal}Q_{\tau}z_{\tau}+2{z_{\tau}}^{\intercal}S_{\tau}\alpha_{\tau} + {\alpha_{\tau}}^{\intercal}R_{\tau}\alpha_{\tau}\big]d\tau+ {z_{T}}^{\intercal}Q_{T}z_{T}\bigg| \mathcal{F}_{h}\bigg]\bigg| \mathcal{F}_{t}\bigg\}\nonumber\allowdisplaybreaks\\
              &= \E\left\{\frac{1}{2}\int_{t}^{h}\left[{z_{\tau}}^{\intercal}Q_{\tau}z_{\tau}+2{z_{\tau}}^{\intercal}S_{\tau}\alpha_{\tau} + {\alpha_{\tau}}^{\intercal}R_{\tau}\alpha_{\tau}\right]d\tau +  \breve{J}_{h,z_{h}}\left(\alpha\right)\bigg| \mathcal{F}_{t}\right\}\nonumber\allowdisplaybreaks\\
              &\geq \E\left\{\frac{1}{2}\int_{t}^{h}\left[{z_{\tau}}^{\intercal}Q_{\tau}z_{\tau}+2{z_{\tau}}^{\intercal}S_{\tau}\alpha_{\tau} + {\alpha_{\tau}}^{\intercal}R_{\tau}\alpha_{\tau}\right]d\tau +  \inf_{\alpha\in \breve{\mathcal{U}}_{h}}\breve{J}_{h,z_{h}}\left(\alpha\right)\bigg| \mathcal{F}_{t}\right\}\nonumber\allowdisplaybreaks\\
              &= \E\left\{\frac{1}{2}\int_{t}^{h}\left[{z_{\tau}}^{\intercal}Q_{\tau}z_{\tau}+2{z_{\tau}}^{\intercal}S_{\tau}\alpha_{\tau} + {\alpha_{\tau}}^{\intercal}R_{\tau}\alpha_{\tau}\right]d\tau +  V(h,z_{h})\bigg| \mathcal{F}_{t}\right\}\nonumber\allowdisplaybreaks\\
              &\geq \inf_{\alpha\in \breve{\mathcal{U}}_{t}}\E\left\{\frac{1}{2}\int_{t}^{h}\left[{z_{\tau}}^{\intercal}Q_{\tau}z_{\tau}+2{z_{\tau}}^{\intercal}S_{\tau}\alpha_{\tau} + {\alpha_{\tau}}^{\intercal}R_{\tau}\alpha_{\tau}\right]d\tau +  V(h,z_{h})\bigg| \mathcal{F}_{t}\right\},\quad \quad dt\otimes d\mathbb{P}-a.s.\label{Append:FirstIneq}
          \end{align} Since the inequality \eqref{Append:FirstIneq} holds for any arbitrary $\alpha\in \breve{\mathcal{U}}_{t}$, for every $ t\in[0,T)$ and $\breve{\xi}\in\left\{\breve{\xi} \in L^{2}_{\mathcal{F}_t}\left(\Omega,\mathbb{R}^{n}\right): \E\big[\breve{\xi}|\mathcal{F}^{0}_{t}\big]=0\right\}$ and $h\in(t,T)$, we have  
          \begin{align}\label{Append:Ineq1}
              V(t,\breve{\xi})&=\inf_{\alpha\in \breve{\mathcal{U}}_{t}}\breve{J}_{t,\breve{\xi}}\left(\alpha\right)\nonumber\allowdisplaybreaks \\
              &\geq \inf_{\alpha\in \breve{\mathcal{U}}_{t}}\E\left\{\frac{1}{2}\int_{t}^{h}\left[{z_{\tau}}^{\intercal}Q_{\tau}z_{\tau}+2{z_{\tau}}^{\intercal}S_{\tau}\alpha_{\tau} + {\alpha_{\tau}}^{\intercal}R_{\tau}\alpha_{\tau}\right]d\tau +  V(h,z_{h})\bigg| \mathcal{F}_{t}\right\},\quad \quad dt\otimes d\mathbb{P}-a.s.
          \end{align}
          \textbf{Step 2.} 
          Consider some arbitrary $t\in[0,T]$, $h\in(t,T)$ and $\alpha\in\breve{\mathcal{U}}_{t}$. By the definition of the value function and considering the characterization of infimum, for any $\epsilon>0$, there exists $\alpha^{\epsilon}\in\breve{\mathcal{U}}_{h}$ such that 
          \begin{equation}\label{Append:InfimCharac}
               \breve{J}_{h,z_{h}}\left(\alpha^{\epsilon}\right) = \frac{1}{2}\E\left[\int_{h}^{T}\left[{z^{\epsilon}_{\tau}}^{\intercal}Q_{\tau}z^{\epsilon}_{\tau}+2{z^{\epsilon}_{\tau}}^{\intercal}S_{\tau}\alpha^{\epsilon}_{\tau} + {\alpha^{\epsilon}_{\tau}}^{\intercal}R_{\tau}\alpha^{\epsilon}_{\tau}\right]d\tau + {z^{\epsilon}_{T}}^{\intercal}Q_{T}z^{\epsilon}_{T}\bigg| \mathcal{F}_{h}\right] < V(h,z_{h}) + \epsilon,\quad dh\otimes d\mathbb{P}-a.s.
          \end{equation} where $z^{\epsilon}\in \mathcal{C}_{\mathcal{F}}\left([h,T],\mathbb{R}^{n}\right)$ satisfies 
                \begin{equation}
                    dz^{\epsilon}_{\tau} = \left[A_{\tau}z^{\epsilon}_{\tau} + B_{\tau}\alpha^{\epsilon}_{\tau}\right]d\tau + D_{\tau} dW_{\tau}, \quad
                    z^{\epsilon}_{h} = z_{h},
               \end{equation} with $z\in \mathcal{C}_{\mathcal{F}}\left([t,T],\mathbb{R}^{n}\right)$ denoting the state trajectory of \eqref{Append:BrevState} corresponding to the control input $\alpha$.

               Consider now the process $\tilde{\alpha}^{\epsilon}$ that is defined as 
               \begin{equation}
                    \tilde{\alpha}^{\epsilon}_{\tau} = \mathbb{1}_{[t,h)}(\tau)\alpha_{\tau} + \mathbb{1}_{[h,T]}(\tau)\alpha^{\epsilon}_{\tau},\quad \forall \tau\in[t,T],\quad h\in(t,T),
               \end{equation} where $\mathbb{1}$ denotes the indicator function. Clearly, $\tilde{\alpha}^{\epsilon}\in L_{\mathcal{F}}^{2}\left( [t,T],\mathbb{R}^{d}\right)$ and $\E\left[\tilde{\alpha}^{\epsilon}_{\tau} \big| \mathcal{F}^{0}_{\tau}\right]=0$, for every $\tau\in[t,T]$, which implies that   $\tilde{\alpha}^{\epsilon}\in\breve{\mathcal{U}}_{t}$. Moreover, the state trajectory $\tilde{z}^{\epsilon}\in \mathcal{C}_{\mathcal{F}}\left([t,T],\mathbb{R}^{n}\right)$ of \eqref{Append:BrevState} corresponding to the control input $\tilde{\alpha}^{\epsilon}$, can be decomposed as 
               \begin{equation}
                   \tilde{z}^{\epsilon}_{\tau} = \mathbb{1}_{[t,h)}(\tau)z_{\tau} + \mathbb{1}_{[h,T]}(\tau)z^{\epsilon}_{\tau},\quad \forall \tau\in[t,T], \quad h\in(t,T).
               \end{equation} Therefore, from \eqref{Append:InfimCharac} and by the law of iterated conditional expectation, for any $\alpha\in\breve{\mathcal{U}}_{t}$,  $\epsilon>0$ and $h\in(t,T)$, we have
               \begin{align}
                   V(t,\breve{\xi}) &\leq \breve{J}_{t,\breve{\xi}}\left(\tilde{\alpha}^{\epsilon}\right)\nonumber\allowdisplaybreaks\\
                   &= \E\left\{\frac{1}{2}\int_{t}^{h}\left[{z_{\tau}}^{\intercal}Q_{\tau}z_{\tau}+2{z_{\tau}}^{\intercal}S_{\tau}\alpha_{\tau} + {\alpha_{\tau}}^{\intercal}R_{\tau}\alpha_{\tau}\right]d\tau +  \breve{J}_{h,z_{h}}\left(\alpha^{\epsilon}\right)\bigg| \mathcal{F}_{t}\right\}\nonumber\allowdisplaybreaks\\
                   &< \E\left\{\frac{1}{2}\int_{t}^{h}\left[{z_{\tau}}^{\intercal}Q_{\tau}z_{\tau}+2{z_{\tau}}^{\intercal}S_{\tau}\alpha_{\tau} + {\alpha_{\tau}}^{\intercal}R_{\tau}\alpha_{\tau}\right]d\tau + V(h,z_{h})\bigg| \mathcal{F}_{t}\right\} + \epsilon,\quad  dt\otimes d\mathbb{P}-a.s.\label{Append:SecIneq}
               \end{align} 
               Since the inequality \eqref{Append:SecIneq} holds for arbitrary $\alpha\in\breve{\mathcal{U}}_{t}$ and $\epsilon>0$, then $\forall h\in(t,T)$, we have
               \begin{equation}\label{Append:Ineq2}
                   V(t,\breve{\xi})\leq \inf_{\alpha\in \breve{\mathcal{U}}}\E\left\{\frac{1}{2}\int_{t}^{h}\left[{z_{\tau}}^{\intercal}Q_{\tau}z_{\tau}+2{z_{\tau}}^{\intercal}S_{\tau}\alpha_{\tau} + {\alpha_{\tau}}^{\intercal}R_{\tau}\alpha_{\tau}\right]d\tau +  V(h,z_{h})\bigg| \mathcal{F}_{t}\right\},\quad dt\otimes d\mathbb{P}-a.s.
               \end{equation} 
                From \eqref{Append:Ineq1} and \eqref{Append:Ineq2}, the equality \eqref{Append:BelmanPrinc} holds. This completes the proof.
\end{proof}    
Similarly to \cite[Theorem 3.4]{XYZ2012}, we obtain that if $(z^{*},\alpha^{*})\in\mathcal{C}_{\mathcal{F}}\left([t,T],\mathbb{R}^{n}\right)\times\breve{\mathcal{U}}_{t}$ is an optimal pair for the problem described by \eqref{Append:BrevCost}-\eqref{Append:BrevState}, then, for every $ h\in(t,T)$, we have 
               \begin{equation}
                   V(h,z^{*}_{h}) = \breve{J}_{h,z^{*}_{h}}\left(\alpha^{*}\right) = \frac{1}{2}\E\left[\int_{h}^{T}\left[{z^{*}_{\tau}}^{\intercal}Q_{\tau}z^{*}_{\tau}+2{z^{*}_{\tau}}^{\intercal}S_{\tau}\alpha^{*}_{\tau} + {\alpha^{*}_{\tau}}^{\intercal}R_{\tau}\alpha^{*}_{\tau}\right]d\tau + {z^{*}_{T}}^{\intercal}Q_{T}z^{*}_{T}\bigg| \mathcal{F}_{h}\right],\quad dt\otimes d\mathbb{P}-a.s.,
               \end{equation}
which demonstrates the validity of the Bellman optimality principle.

    \subsection{Stochastic HJB Equation and Optimal Control}
Following \cite{Peng1992,Qiu2018,Moon2022,Tang2015}, and temporarily disregarding the constraint on the admissible control set $\breve{\mathcal{U}}_{t}$, the value function \eqref{Append:ValueFunc} satisfies the stochastic HJB equation given by
\begin{equation}\label{Append:HJB}
    \begin{cases} 
        \mathcal{V}_{t}(t,x) =  -\displaystyle\inf_{ a \in \mathbb{R}^{d}}H(t,x,a,\mathcal{V}_{x}(t,x),\mathcal{V}_{xx}(t,x),\Sigma_{x}(t,x)) + \Sigma(t,x) dW^{0}_{t},\quad  (t,x)\in[0,T)\times \mathbb{R}^n\\
        \mathcal{V}(T,x) = \frac{1}{2}x^{\intercal}Q_{T}x,\quad  x\in \mathbb{R}^n
    \end{cases}
\end{equation}
where $\mathcal{V}_{t}$ denotes the partial derivative of $\mathcal{V}$ with respect to $t$, $\mathcal{V}_{x}$ and $\Sigma_{x}$ respectively represent the gradients of $\mathcal{V}$ and $\Sigma$ with respect to $x$, and $\mathcal{V}_{xx}$ denotes the Hessian matrix of $\mathcal{V}$ with respect to $x$. Moreover, $H$ is the stochastic Hamiltonian defined for $(t,x,\alpha,C,G,E)\in [0,T]\times \mathbb{R}^{n}\times\mathbb{R}^{d}\times \mathbb{R}^{n}\times \mathbb{S}^{n}\times\mathbb{R}^{n}$ by
\begin{equation}
    H(t,x,a,C,G,E):= (A_{t} x + B_{t}a)^{\intercal}C + \frac{1}{2}D^{\intercal}_{t}GD_{t} + D^{\intercal}_{t}E + \frac{1}{2}\left(x^{\intercal}Q_{t}x + 2{x}^{\intercal}S_{t}a + {a}^{\intercal}R_{t}a\right).
\end{equation} 
Under \Cref{A2}, the above stochastic HJB equation is well-posed \cite{Peng1992}. By considering the verification theorems \cite[Subsection 3.2]{Peng1992} and \cite[Theorem 2]{Moon2022} and by following the same procedure as in \cite[Subsection 2.3]{Moon2022}, for any $t\in[0,T],\,\breve{\xi}\in\left\{\breve{\xi} \in L^{2}_{\mathcal{F}_t}\left(\Omega,\mathbb{R}^{n}\right): \E\big[\breve{\xi}|\mathcal{F}^{0}_{t}\big]=0\right\}$, the value function \eqref{Append:ValueFunc} for problem \eqref{Append:BrevCost}-\eqref{Append:BrevState} has the following analytical form 
\begin{equation}
    V(t,\breve{\xi}) = \frac{1}{2}\breve{\xi}^{\intercal}P_{t}\breve{\xi},\quad dt\otimes d\mathbb{P}-a.s.
\end{equation}
Moreover, the optimal control $\alpha^{*}$ for problem \eqref{Append:BrevCost}-\eqref{Append:BrevState} is given by 
\begin{equation} \label{Append:OptimalControl}
    \alpha^{*}_{t} = R^{-1}_{t}\left(B^{\intercal}_{t}P_{t} + S^{\intercal}\right)z^{*}_{t},\quad dt\otimes d\mathbb{P}-a.s,
\end{equation} and the corresponding optimal trajectory $z^{*}$ satisfies
\begin{equation}
        dz^{*}_{t} = \left(A_{t} - B_{t}R^{-1}_{t}{(S_{t}+P_{t}B_{t})}^{\intercal}\right)z^{*}_{t}dt + D_{t} dW_{t}, \quad z^{*}_{0} = \xi - \E[\xi],
    \end{equation}
where the process $P$ is the solution of the stochastic Riccati equation given by
\begin{align}\label{Append:SRE}
    dP_{t} =& -\left[P_{t}\left(A_{t} - B_{t}R^{-1}_{t}S^{\intercal}_{t}\right) + \left(A_{t} - B_{t}R^{-1}_{t}S^{\intercal}_{t}\right)^{\intercal}P_{t} - P_{t}B_{t}R_{t}^{-1}B_{t}^{\intercal}P_{t} + Q_{t} - S_{t}R_{t}^{-1}S_{t}^{\intercal}\right]dt\nonumber\\
             &+\Psi_{t}^{0}dW_{t}^{0},\quad P_{T} = Q_{T},\allowdisplaybreaks
\end{align}
which coincides with the one obtained using the  variational method and given by \eqref{SRE}.

We observe that the obtained optimal control $\alpha^*$ given by \eqref{Append:OptimalControl} belongs to the constrained admissible set $\breve{\mathcal{U}}_{t}:= \Big\{ \alpha \in L_{\mathcal{F}}^{2}\left( [t,T],\mathbb{R}^{d}\right): \E\left[\alpha_{\tau} \big| \mathcal{F}^{0}_{\tau}\right]=0,\,d\tau\otimes d\mathbb{P}-a.s. \Big\}$. For further details on the verification of this constraint, we refer the reader to the proof of \Cref{OptimPrincMFT}. Indeed, by construction, the structure of the problem defined by \eqref{Append:BrevCost}–\eqref{Append:BrevState} ensures that the optimal control automatically satisfies the admissibility constraint.

\section{Well-Posedness of McKean-Vlasov FBSDEs with Stochastic Coefficients}\label{AppendixC}
We aim to show that according to \cite[Theorem 1, pp.~3863]{Ahuja2019}, one may establish the well-posedness of the linear Mckean-Valsov FBSDEs with random coefficients given as 
\begin{equation}\label{FBSDE:WellPos}
\begin{cases}
dX_{t} = \left( \hat{A}_{t}X_{t} + \hat{B}_{t}\mathbb{E}\left[X_{t}|\mathcal{F}_{t}^0\right] + \hat{C}_{t}Y_{t} + \hat{S}_{t}\right) dt  + \hat{D}_{t}dW_{t} + \hat{D}^{0}_{t}dW^{0}_{t}, \\
dY_{t} = -\left( \hat{A}^{\intercal}_{t} Y_{t} + \hat{B}^{\intercal}_{t} \mathbb{E}\left[Y_{t}|\mathcal{F}_{t}^0\right] + \tilde{C}_{t}X_{t} + \tilde{F}_{t}\mathbb{E}\left[X_{t}|\mathcal{F}_{t}^0\right] + \tilde{S}_{t}\right) dt + Z_{t} dW_{t} + Z^{0}_{t}dW^{0}_{t}, \\
X_0 = x, \quad Y_T = \hat{G}X_T + \tilde{G}\mathbb{E}[X_T|\mathcal{F}_{T}^0].
\end{cases}
\end{equation}
To this end, we impose the following conditions on the coefficients of the above FBSDE.
\begin{cond1}\thlabel{H1}
    $\hat{A}, \hat{B}, \tilde{C}, \tilde{F}\in L_{\mathcal{F}^{0}}^{\infty}\left( [0,T],\mathbb{R}^{n\times n}\right)$, $\hat{C}\in L_{\mathcal{F}^{0}}^{\infty}\left( [0,T],\mathbb{S}^{n}\right)$, $\hat{S},\tilde{S},\hat{D},\hat{D}^{0}\in L_{\mathcal{F}^{0}}^{2}\left([0,T],\mathbb{R}^{n}\right)$, $\hat{G},\tilde{G}\in L_{\mathcal{F}^{0}_{T}}^{\infty}\left(\mathbb{S}^{n}\right)$. 
\end{cond1}
\begin{cond2}\thlabel{H2}
    $-\hat{C}_{t}\geq 0,\tilde{C}_{t}\geq 0, \tilde{F}_{t} + \tilde{C}_{t}\geq 0, \,dt\otimes d\mathbb{P}-a.s$, and $\hat{G}\geq 0, \tilde{G}+\hat{G}\geq 0 \,d\mathbb{P}-a.s$.
\end{cond2}

\begin{theorem}\label{thm:wellposed-gen-FBSDE}
  Suppose that \hyperref[H1]{\thnameref{H1}} and \hyperref[H2]{\thnameref{H2}} hold. Then, the linear McKean-Vlasov FBSDEs given by \eqref{FBSDE:WellPos} admit a unique solution. 
\end{theorem}
\begin{proof}
Define
\begin{align*}
    &b\left(t, X_{t}, \mathbb{E}\left[X_{t}|\mathcal{F}_{t}^0\right], Y_{t}, \mathbb{E}\left[Y_{t}|\mathcal{F}_{t}^0\right], Z_{t}, Z^{0}_{t}\right):= \hat{A}_{t}X_{t} + \hat{B}_{t}\mathbb{E}\left[X_{t}|\mathcal{F}_{t}^0\right] + \hat{C}_{t}Y_{t} + \hat{S}_{t},\allowdisplaybreaks\\
    &f\left(t, X_{t}, \mathbb{E}\left[X_{t}|\mathcal{F}_{t}^0\right], Y_{t}, \mathbb{E}\left[Y_{t}|\mathcal{F}_{t}^0\right], Z_{t}, Z^{0}_{t}\right):= -\left( \hat{A}^{\intercal}_{t} Y_{t} + \hat{B}^{\intercal}_{t} \mathbb{E}\left[Y_{t}|\mathcal{F}_{t}^0\right] + \tilde{C}_{t}X_{t} + \tilde{F}_{t}\mathbb{E}\left[X_{t}|\mathcal{F}_{t}^0\right] + \tilde{S}_{t}\right),\allowdisplaybreaks\\
    &\sigma\left(t, X_{t}, \mathbb{E}\left[X_{t}|\mathcal{F}_{t}^0\right], Y_{t}, \mathbb{E}\left[Y_{t}|\mathcal{F}_{t}^0\right], Z_{t}, Z^{0}_{t}\right):= \hat{D}_{t},\allowdisplaybreaks\\
    &\tilde{\sigma}\left(t, X_{t}, \mathbb{E}\left[X_{t}|\mathcal{F}_{t}^0\right], Y_{t}, \mathbb{E}\left[Y_{t}|\mathcal{F}_{t}^0\right], Z_{t}, Z^{0}_{t}\right):= \hat{D}^{0}_{t},\allowdisplaybreaks\\
    &g\left(X_{T}, \mathbb{E}\left[X_{T}|\mathcal{F}_{T}^0\right]\right):= \hat{G}X_T + \tilde{G}\mathbb{E}[X_T|\mathcal{F}_{T}^0].
\end{align*} 
We aim to show that, under \hyperref[H1]{\thnameref{H1}}-\hyperref[H2]{\thnameref{H2}}, the corresponding lifting functionals $B, \Sigma, \Sigma^{0}, F, G$ of \eqref{FBSDE:WellPos} given by 
\begin{equation}
    \begin{aligned}
        &B\left(t, X_{t}, Y_{t}, Z_{t}, Z^{0}_{t}\right) = b\left(t, X_{t}, \mathbb{E}\left[X_{t}|\mathcal{F}_{t}^0\right], Y_{t}, \mathbb{E}\left[Y_{t}|\mathcal{F}_{t}^0\right], Z_{t}, Z^{0}_{t}\right),\allowdisplaybreaks\\
        &F\left(t, X_{t}, Y_{t}, Z_{t}, Z^{0}_{t}\right) =f\left(t, X_{t}, \mathbb{E}\left[X_{t}|\mathcal{F}_{t}^0\right], Y_{t}, \mathbb{E}\left[Y_{t}|\mathcal{F}_{t}^0\right], Z_{t}, Z^{0}_{t}\right),\allowdisplaybreaks\\
        &\Sigma\left(t, X_{t}, Y_{t}, Z_{t}, Z^{0}_{t}\right) = \sigma\left(t, X_{t}, \mathbb{E}\left[X_{t}|\mathcal{F}_{t}^0\right], Y_{t}, \mathbb{E}\left[Y_{t}|\mathcal{F}_{t}^0\right], Z_{t}, Z^{0}_{t}\right),\allowdisplaybreaks\\
        &\Sigma^{0}\left(t, X_{t}, Y_{t}, Z_{t}, Z^{0}_{t}\right) = \tilde{\sigma}\left(t, X_{t}, \mathbb{E}\left[X_{t}|\mathcal{F}_{t}^0\right], Y_{t}, \mathbb{E}\left[Y_{t}|\mathcal{F}_{t}^0\right], Z_{t}, Z^{0}_{t}\right),\allowdisplaybreaks\\
        &G\left(X_{T}\right) = g\left(X_{T}, \mathbb{E}\left[X_{T}|\mathcal{F}_{T}^0\right]\right),
    \end{aligned}
\end{equation}
satisfy assumptions \cite[(A1)-(A3)]{Ahuja2019}, where $\mathcal{F}_{t} := \sigma\left\{W_{s},W^{0}_{s}; 0 \le s \le t\right\}$ and $\mathcal{G}_{t} :=\mathcal{F}_{t}^0 = \sigma\left\{W^{0}_{s}; 0 \le s \le t\right\}$. 

It is clear that, under \hyperref[H1]{\thnameref{H1}},  $\mbox{for}\,\, \phi = b,f,\sigma,\tilde{\sigma}$, we have
\begin{equation}\label{FirstCond}
    \mathbb{E}\left(\int^{T}_{0}\Vert \phi(t,0,0,0,0,0,0)\Vert^{2}dt\right)<\infty.
\end{equation} 
Therefore, assumptions \cite[(A1)-(A2)]{Ahuja2019} follow from \hyperref[H1]{\thnameref{H1}}. 

We now shift our focus to assumption in \cite[(A3)]{Ahuja2019}. For any tuples $(X_{T}, Y_{T}, Z_{T}, Z^{0}_{T}) \mbox{ and } (X'_{T}, Y'_{T}, Z'_{T}, Z^{0'}_{T})$, define $\Delta X_{T}:= X_{T} - X'_{T}$ and $\Delta g := g\left(X_{T}, \mathbb{E}\left[X_{T}|\mathcal{F}_{T}^0\right]\right) - g\left(X'_{T}, \mathbb{E}\left[X'_{T}|\mathcal{F}_{T}^0\right]\right)$. Then, we have 
\begin{align}\label{ZeroFirst}
    \mathbb{E}\left[\left\Vert\Delta g\right\Vert^{2}\right] &= \mathbb{E}\left[\left\Vert \hat{G}\Delta X_T + \mathbb{E}[\tilde{G}\Delta X_T|\mathcal{F}_{T}^0]\right\Vert^{2}\right]\nonumber\allowdisplaybreaks\\
    &\leq 2\mathbb{E}\left[\left\Vert \hat{G}\Delta X_T\right\Vert^{2}\right] + 2\mathbb{E}\left[\left\Vert \mathbb{E}[\tilde{G}\Delta X_T|\mathcal{F}_{T}^0]\right\Vert^{2}\right]\nonumber\allowdisplaybreaks\\
    &\leq 2\mathbb{E}\left[\left\Vert \hat{G}\right\Vert^{2} \left\Vert\Delta X_T\right\Vert^{2}\right] + 2\mathbb{E}\left[\left(\mathbb{E}\left[\left\Vert\tilde{G}\Delta X_T\right\Vert|\mathcal{F}_{T}^0\right]\right)^{2}\right]\allowdisplaybreaks\\
    &\leq 2\Vert \hat{G} \Vert^{2}_{L_{\mathcal{F}^{0}_{T}}^{\infty}}\mathbb{E}\left[\left\Vert\Delta X_T\right\Vert^{2}\right] + 2\mathbb{E}\left(\mathbb{E}\left[\left\Vert\tilde{G}\right\Vert^{2}|\mathcal{F}_{T}^0\right]\mathbb{E}\left[\left\Vert\Delta X_T\right\Vert^{2}|\mathcal{F}_{T}^0\right]\right)\nonumber\allowdisplaybreaks\\
    &\leq 2\Vert \hat{G} \Vert^{2}_{L_{\mathcal{F}^{0}_{T}}^{\infty}}\mathbb{E}\left[\left\Vert\Delta X_T\right\Vert^{2}\right] + 2\Vert \tilde{G} \Vert^{2}_{L_{\mathcal{F}^{0}_{T}}^{\infty}}\mathbb{E}\left[\mathbb{E}\left[\left\Vert\Delta X_T\right\Vert^{2}|\mathcal{F}_{T}^0\right]\right]\nonumber\allowdisplaybreaks\\
    &=2\left(\Vert \hat{G} \Vert^{2}_{L_{\mathcal{F}^{0}_{T}}^{\infty}} + \Vert \tilde{G} \Vert^{2}_{L_{\mathcal{F}^{0}_{T}}^{\infty}}\right)\mathbb{E}\left[\left\Vert\Delta X_T\right\Vert^{2}\right]\nonumber\allowdisplaybreaks,
\end{align}
where to derive the third inequality, the Cauchy-Schwarz inequality was used. 

Moreover, for any $t \in [0, T]$ and any tuples $(X_{t}, Y_{t}, Z_{t},Z^{0}_{t}), (X'_{t}, Y'_{t}, Z'_{t}, Z^{0'}_{t})$, define $\Delta X_{t} := X_{t} - X'_{t}, \Delta Y_{t} := Y_{t} - Y'_{t}$ and for $\phi = b,f,\sigma,\tilde{\sigma}$, define
\begin{gather*}
    \Delta \phi_{t}:=\phi\left(t, X_{t}, \mathbb{E}\left[X_{t}|\mathcal{F}_{t}^0\right], Y_{t}, \mathbb{E}\left[Y_{t}|\mathcal{F}_{t}^0\right], Z_{t}, Z^{0}_{t}\right) - \phi\left(t, X'_{t}, \mathbb{E}\left[X'_{t}|\mathcal{F}_{t}^0\right], Y'_{t}, \mathbb{E}\left[Y'_{t}|\mathcal{F}_{t}^0\right], Z'_{t}, Z^{0'}_{t}\right).
\end{gather*} 
Similarly to \eqref{ZeroFirst}, we obtain
\begin{align*}
    \mathbb{E}\left[\left\Vert\Delta b_{t}\right\Vert^{2}\right]&\leq 3\left(\Vert \hat{A} \Vert^{2}_{L_{\mathcal{F}^{0}}^{\infty}} + \Vert \hat{B} \Vert^{2}_{L_{\mathcal{F}^{0}}^{\infty}} \right)\mathbb{E}\left[\left\Vert\Delta X_t\right\Vert^{2}\right] + 3\left\Vert (-\hat{C})^{\frac{1}{2}} \right\Vert^{2}_{L_{\mathcal{F}^{0}}^{\infty}}\mathbb{E}\left[\left\Vert (-\hat{C}_{t})^{\frac{1}{2}}\Delta Y_t\right\Vert^{2}\right],\allowdisplaybreaks\\
    \mathbb{E}\left[\left\Vert\Delta f_{t}\right\Vert^{2}\right]&\leq 4\left(\Vert \tilde{C} \Vert^{2}_{L_{\mathcal{F}^{0}}^{\infty}} + \Vert \tilde{F} \Vert^{2}_{L_{\mathcal{F}^{0}}^{\infty}} \right)\mathbb{E}\left[\left\Vert\Delta X_t\right\Vert^{2}\right] + 4\left(\Vert \hat{A} \Vert^{2}_{L_{\mathcal{F}^{0}}^{\infty}} + \Vert \hat{B} \Vert^{2}_{L_{\mathcal{F}^{0}}^{\infty}}\right)\mathbb{E}\left[\left\Vert\Delta Y_t\right\Vert^{2}\right],\allowdisplaybreaks\\
    \mathbb{E}\left[\left\Vert\Delta \sigma_{t}\right\Vert^{2}\right]&=0,\allowdisplaybreaks\\
    \mathbb{E}\left[\left\Vert\Delta \tilde{\sigma}_{t}\right\Vert^{2}\right]&=0,
\end{align*}
where the positive definiteness of $-\hat{C}_{t}$, imposed in \hyperref[H2]{\thnameref{H2}}, ensures that $(-\hat{C}_{t})^{\frac{1}{2}}$ is well defined. Therefore, there exist a positive constant $K$ and a uniformly bounded linear map $c_t:\mathbb{R}^n \to \mathbb{R}^n$, defined by $c_t(Y)=(-\hat{C}_{t})^{\frac{1}{2}}Y$, such that
\begin{align}\label{SecCond}
    \mathbb{E}\left[\left\Vert\Delta b_{t}\right\Vert^{2}\right]&\leq K\mathbb{E}\left[\left\Vert\Delta X_t\right\Vert^{2} + \left\Vert c_{t}\left(\Delta Y_t\right)\right\Vert^{2}\right]\nonumber\allowdisplaybreaks\\
    \mathbb{E}\left[\left\Vert\Delta \sigma_{t}\right\Vert^{2}\right]&\leq K\mathbb{E}\left[\left\Vert\Delta X_t\right\Vert^{2} + \left\Vert c_{t}\left(\Delta Y_t\right)\right\Vert^{2}\right]\nonumber\allowdisplaybreaks\\
    \mathbb{E}\left[\left\Vert\Delta \tilde{\sigma}_{t}\right\Vert^{2}\right]&\leq K\mathbb{E}\left[\left\Vert\Delta X_t\right\Vert^{2} + \left\Vert c_{t}\left(\Delta Y_t\right)\right\Vert^{2}\right]\allowdisplaybreaks\\
    \mathbb{E}\left[\left\Vert\Delta f_{t}\right\Vert^{2}\right]&\leq K\mathbb{E}\left[\left\Vert\Delta X_t\right\Vert^{2} + \left\Vert\Delta Y_t\right\Vert^{2}\right]\nonumber \allowdisplaybreaks\\
    \mathbb{E}\left[\left\Vert\Delta g\right\Vert^{2}\right]&\leq K\mathbb{E}\left[\left\Vert\Delta X_T\right\Vert^{2}\right]\nonumber.
\end{align} 
Furthermore, we have
\begin{align}
    \mathbb{E}\left[\left\langle \Delta g,\Delta X_T\right\rangle\right] &= \mathbb{E}\left[\Delta X_T^{\intercal}\left(\hat{G}\Delta X_T + \tilde{G}\mathbb{E}\left[\Delta X_T|\mathcal{F}_{T}^0\right]\right)\right]\allowdisplaybreaks \notag\\
    &=\mathbb{E}\left[\Delta X_T^{\intercal}\hat{G}\Delta X_T\right] + \mathbb{E}\left[\Delta X_T^{\intercal} \tilde{G}\mathbb{E}\left[\Delta X_T|\mathcal{F}_{T}^0\right]\right]\allowdisplaybreaks \notag\\
    &=\mathbb{E}\left[\Delta X_T^{\intercal}\hat{G}\Delta X_T\right] + \mathbb{E}\left(\mathbb{E}\left[\Delta X_T^{\intercal} \tilde{G}\mathbb{E}\left[\Delta X_T|\mathcal{F}_{T}^0\right]\Big|\mathcal{F}_{T}^0\right]\right)\notag\\
    &=\mathbb{E}\left[\Delta X_T^{\intercal}\hat{G}\Delta X_T\right] - \mathbb{E}\left(\mathbb{E}\left[\Delta X_T|\mathcal{F}_{T}^0\right]^{\intercal}\hat{G}\mathbb{E}\left[\Delta X_T|\mathcal{F}_{T}^0\right]\right) + \mathbb{E}\left(\mathbb{E}\left[\Delta X_T|\mathcal{F}_{T}^0\right]^{\intercal}(\tilde{G}+\hat{G})\mathbb{E}\left[\Delta X_T|\mathcal{F}_{T}^0\right]\right)\allowdisplaybreaks \notag\\
    &=\mathbb{E}\left[\left(\Delta X_T - \mathbb{E}\left[\Delta X_T|\mathcal{F}_{T}^0\right]\right)^{\intercal}\hat{G}\left(\Delta X_T - \mathbb{E}\left[\Delta X_T|\mathcal{F}_{T}^0\right]\right)\right]\allowdisplaybreaks + \mathbb{E}\left(\mathbb{E}\left[\Delta X_T|\mathcal{F}_{T}^0\right]^{\intercal}(\tilde{G} + \hat{G})\mathbb{E}\left[\Delta X_T|\mathcal{F}_{T}^0\right]\right)\notag\\
    &\geq 0,\label{B2b1}
\end{align} 
where the last inequality follows from \hyperref[H2]{\thnameref{H2}}. 
 
Similarly, we have
\begin{align*}
    \mathbb{E}\left[\left\langle \Delta f_{t},\Delta X_{t} \right\rangle + \left\langle \Delta b_{t},\Delta Y_{t}\right\rangle  \right] &= - \mathbb{E}\left[\Delta X_{t}^{\intercal}\tilde{C}_{t}\Delta X_{t} \right] - \mathbb{E}\left( \mathbb{E}\left[\Delta X_t|\mathcal{F}_{t}^0\right]^{\intercal}\tilde{F}_{t} \mathbb{E}\left[\Delta X_t|\mathcal{F}_{t}^0\right]\right) - \mathbb{E}\left[ \Delta Y_{t}^{\intercal}(-\hat{C}_{t})\Delta Y_{t} \right] \\
     &= - \mathbb{E}\left[\Delta X_{t}^{\intercal}\tilde{C}_{t}\Delta X_{t} \right] + \mathbb{E}\left( \mathbb{E}\left[\Delta X_t|\mathcal{F}_{t}^0\right]^{\intercal}\tilde{C}_{t} \mathbb{E}\left[\Delta X_t|\mathcal{F}_{t}^0\right]\right)\allowdisplaybreaks\\
     &\quad - \mathbb{E}\left( \mathbb{E}\left[\Delta X_t|\mathcal{F}_{t}^0\right]^{\intercal}(\tilde{F}_{t}+\tilde{C}_{t}) \mathbb{E}\left[\Delta X_t|\mathcal{F}_{t}^0\right]\right) - \mathbb{E}\left[ \Delta Y_{t}^{\intercal}(-\hat{C}_{t})\Delta Y_{t} \right] \allowdisplaybreaks\\
     &= - \mathbb{E}\left[\left(\Delta X_{t} - \mathbb{E}\left[\Delta X_t|\mathcal{F}_{t}^0\right]\right)^{\intercal}\tilde{C}_{t}\left(\Delta X_{t} - \mathbb{E}\left[\Delta X_t|\mathcal{F}_{t}^0\right]\right)\right] \allowdisplaybreaks\\
     &\quad - \mathbb{E}\left( \mathbb{E}\left[\Delta X_t|\mathcal{F}_{t}^0\right]^{\intercal}(\tilde{F}_{t}+\tilde{C}_{t}) \mathbb{E}\left[\Delta X_t|\mathcal{F}_{t}^0\right]\right) - \mathbb{E}\left[ \Delta Y_{t}^{\intercal}(-\hat{C}_{t})\Delta Y_{t} \right]\\
     &\leq - \mathbb{E}\left[ \Delta Y_{t}^{\intercal}(-\hat{C}_{t})\Delta Y_{t} \right]\nonumber\allowdisplaybreaks\\
     &= -\mathbb{E}\left[\left\Vert c_{t}\left(\Delta Y_t\right)\right\Vert^{2}\right],
\end{align*} 
where the last inequality is a consequence of \hyperref[H2]{\thnameref{H2}}.  Since, $\Delta \sigma_{t} = 0$ and $\Delta \tilde{\sigma}_{t} = 0$, then
\begin{align}\label{B2b2}
    \mathbb{E}\left[\left\langle \Delta f_{t},\Delta X_{t} \right\rangle + \left\langle \Delta b_{t},\Delta Y_{t}\right\rangle + \left\langle \Delta \sigma_{t},\Delta Z_{t} \right\rangle + \left\langle \Delta \tilde{\sigma}_{t},\Delta Z^{0'}_{t} \right\rangle\right] &\leq -\mathbb{E}\left[\left\Vert c_{t}\left(\Delta Y_t\right)\right\Vert^{2}\right].
\end{align} 
Finally, from \eqref{SecCond}-\eqref{B2b2}, assumption \cite[(A3)]{Ahuja2019} is satisfied for the FBSDEs given by \eqref{FBSDE:WellPos}. Specifically, for any event $E\in \mathcal{G}_{t}$, from \eqref{SecCond}, we obtain  
\begin{align*}
    \mathbb{E}\left[\mathbb{1}_{E}\left\Vert\Delta B_{t}\right\Vert^{2}\right]&= \mathbb{E}\left[\mathbb{1}_{E}\left\Vert\Delta b_{t}\right\Vert^{2}\right] \allowdisplaybreaks\\
    &= \mathbb{E}\left(\mathbb{E}\left[\mathbb{1}_{E}\left\Vert\Delta b_{t}\right\Vert^{2}\big|\mathcal{G}_{t}\right]\right)\allowdisplaybreaks\\
    &= \mathbb{E}\left(\mathbb{1}_{E}\mathbb{E}\left[\left\Vert\Delta b_{t}\right\Vert^{2}\big|\mathcal{G}_{t}\right]\right)\allowdisplaybreaks\\
    &\leq \mathbb{E}\left(\mathbb{1}_{E}\mathbb{E}\left[K\left(\left\Vert\Delta X_t\right\Vert^{2} + \left\Vert c_{t}\left(\Delta Y_t\right)\right\Vert^{2}\right)\Big|\mathcal{G}_{t}\right]\right)\allowdisplaybreaks\\
    &= K\mathbb{E}\left[\mathbb{1}_{E}\left(\left\Vert\Delta X_t\right\Vert^{2} + \left\Vert c_{t}\left(\Delta Y_t\right)\right\Vert^{2}\right)\right],
\end{align*}
where, by \hyperref[H1]{\thnameref{H1}}, $c_t(Y_t)=(-\hat{C}_{t})^{\frac{1}{2}}Y_t$ satisfies \cite[(A3)-(a)]{Ahuja2019}. The remaining conditions \cite[(A3)-(b)]{Ahuja2019} and  \cite[(A3)-(c)]{Ahuja2019} also follow similarly. Consequently, under \hyperref[H1]{\thnameref{H1}}-\hyperref[H2]{\thnameref{H2}}, \cite[Theorem 1, pp.~3863]{Ahuja2019} guarantees that the FBSDE \eqref{FBSDE:WellPos} admits a unique solution.
\end{proof}

\begin{corollary}\label{wellposed_our_FBSDE} Suppose that \Cref{A2} holds. Then, the linear McKean-Vlasov FBSDEs given by \eqref{stoch-max-1} and \eqref{stoch-max-2} admit a unique solution. 
\end{corollary}
\begin{proof}
We note that the FBSDEs given by \eqref{stoch-max-1}-\eqref{stoch-max-2}, with the expression for $u^\circ$ substituted, follow the same form as \eqref{FBSDE:WellPos}. Specifically, 
$X_{t}:=x^\circ_{t}, Y_{t}:=y^\circ_{t},    Z_{t}:=z^\circ_{t},
    Z^{0}_{t}:=z^{0,\circ}_{t},
    \hat{A}_{t}:= A_{t} - B_{t}R^{-1}_{t}S^{\intercal}_{t},
    \hat{B}_{t}:=F_{t} + B_{t}R^{-1}_{t}S^{\intercal}_{t}H,
    \hat{C}_{t}:=- B_{t}R^{-1}_{t}B^{\intercal}_{t},
    \hat{S}_{t}:=- B_{t}R^{-1}_{t}\varpi_{t} + b_{t},
    \hat{D}_{t}:= D_{t},
    \hat{D}^{0}_{t}:= D^{0}_{t},
    \tilde{C}_{t}:=Q_{t} - S_{t}R^{-1}_{t}S^{\intercal}_{t},
    \tilde{S}_{t}:=- \bar{S}_{t}R^{-1}_{t}\varpi_{t} + \bar{\zeta}_{t},
    \hat{G}:= Q_{T},
    \tilde{G}= \bar{Q}_{T} - Q_{T},
    \tilde{F}_{t}:=\bar{Q}_{t} -  \bar{S}_{t}R^{-1}_{t}\bar{S}^{\intercal}_{t} - Q_{t} + S_{t}R^{-1}_{t}S^{\intercal}_{t}= \left(I_{n} - H\right)^{\intercal}\left[Q - SR^{-1}S^{\intercal}\right]\left(I_{n} - H\right) - \left(Q_{t} - S_{t}R^{-1}_{t}S^{\intercal}_{t}\right).
$ Then, by \Cref{A2} and the conditions imposed on the coefficients in \eqref{stoch-max-1}-\eqref{stoch-max-2} (see the description of {\color{black}\thnameref{MTProb}}), the conditions \hyperref[H1]{\thnameref{H1}} and \hyperref[H2]{\thnameref{H2}} are satisfied straightforwardly. Consequently, according to \Cref{thm:wellposed-gen-FBSDE}, the FBSDEs \eqref{stoch-max-1}-\eqref{stoch-max-2} admit a unique solution.
\end{proof}

\end{document}